\documentclass[12pt]{article}
\usepackage{amsmath,amsthm,amssymb}
\usepackage{graphicx}

\usepackage{subfigure}
\usepackage{changepage}

\def\R{\mathbb{R}}
\def\C{\mathbb{C}}
\def\P{\mathbb{P}}
\def\E{\mathbb{E}}
\def\N{\mathbb{N}}
\def\Z{\mathbb{Z}}

\DeclareMathOperator{\Poisson}{Poisson}

\DeclareMathOperator{\polydim}{PolyDim}

\newtheorem{definition}{Definition}[section]
\newtheorem{theorem}{Theorem}[section]
\newtheorem{proposition}[theorem]{Proposition}
\newtheorem{lemma}[theorem]{Lemma}

\newtheorem{corollary}[theorem]{Corollary}

\def\eps{\varepsilon}

\newcommand{\df}[1]{{\bf #1}}

\renewcommand\Box[2]{\text{Box($#1$, $#2$)}}
\newcommand\Cyl[2]{\text{Cyl($#1$, $#2$)}}

\newcommand{\Vt}{V_{-}}
\newcommand{\Vs}{V_0}
\newcommand{\Vh}{V_{+}}

\newcommand{\diam}{\text{Diam}}

\newcommand{\vol}{\text{Vol}}
\newcommand{\p}{{\mathbb P}}

\newcommand{\CZ}{{\cal Z}}
\newcommand{\pd}{U^{\text{diff}}}

\title{Phase Transitions in Gravitational Allocation}
\author{Sourav Chatterjee\thanks{U.C. Berkeley. Supported by NSF grant DMS-0707054 and a Sloan Research Fellowship.} \and Ron Peled\thanks{New York University. Partially completed during stay at the Institut Henri Poincare - Centre Emile Borel. Research supported by NSF Grant OISE 0730136.} \and Yuval Peres\thanks{Microsoft Research.} \and Dan Romik\thanks{Hebrew University of Jerusalem. Supported by the Israel Science Foundation (ISF) grant number 1051/08.}}
\begin{document}
\changepage{}{13pt}{}{}{}{}{}{}{}

\maketitle 

\begin{abstract}
Given a Poisson point process of unit masses (``stars'') in dimension $d\ge3$, Newtonian gravity partitions space into domains of attraction (cells) of equal volume. In earlier work, we showed the diameters of these cells have exponential tails. Here we analyze the quantitative geometry of the cells and show that their large deviations occur at the stretched-exponential scale. More precisely, the probability that mass $\exp(-R^{\gamma})$ in a cell travels distance $R$ decays like $\exp(-R^{f_d(\gamma)})$ where we identify the functions $f_d(\cdot)$ exactly. These functions are piecewise smooth and the discontinuities of $f_d'$ represent phase transitions. In dimension $d=3$, the large deviation is due to a ``distant attracting galaxy'' but  a phase transition occurs when $f_3(\gamma)=1$ (at that point, the fluctuations due to individual stars dominate). When $d\ge 5$, the large deviation is due to a thin tube (a ``wormhole'') along which the star density increases monotonically, until the point $f_d(\gamma)=1$ (where again fluctuations due to individual stars dominate). In dimension 4 we find a double phase transition, where the transition between low-dimensional behavior (attracting galaxy) and high-dimensional behavior (wormhole) occurs at $\gamma=\frac{4}{3}$.

As consequences, we  determine the tail behavior of the distance from a star to a uniform point in its cell,
and prove a sharp lower bound for the tail probability of the cell's diameter, matching our earlier upper bound.


\end{abstract}

\newpage
\changepage{}{-13pt}{}{}{}{}{}{}{}
\tableofcontents

\newpage

\begin{section}{Introduction}
\begin{subsection}{The main results}

Let $d\ge 3$ and let $\CZ$ be a standard Poisson point process (``the stars'') in $\R^d$. The (random) \df{gravitational force field function} $F(x)$ is defined by
$$ F(x) = \sum_{x\in\CZ,\ |z-x|\uparrow} \frac{z-x}{|z-x|^d} $$
(the summands are ordered by increasing distance from $x$; recall that in \cite{CPPR07} it is proved that the sum converges conditionally a.s.\ when the summands are ordered in this way). Then for each $z\in\CZ$ we denote by $B(z)$ its basin of attraction (also called its cell) in the \df{gravitational allocation} defined in \cite{CPPR07}. Loosely speaking, $B(z)$ is the set of points which flow into $z$ under the gravitational flow
$$ \dot{x} = F(x), $$

Denote also by $\psi_{\CZ}$ the \df{allocation mapping}, given by
 \begin{equation*}
  \psi_{\CZ}(x) = \begin{cases}z& x\in B(z)\text{ for some $z\in\CZ$,}\\ \infty& x\notin\cup_{z\in\CZ}B(z).\end{cases}
 \end{equation*}

In \cite{CPPR07} we showed that all the cells $B(z)=\psi_\CZ^{-1}(z)$ have volume 1 (essentially a consequence of the divergence theorem, first discovered in a different context in \cite{ST06}). This means that gravitational allocation is 
a fair and translation-equivariant allocation rule. Not only is it a rather natural construction, but we also analyzed it and showed that it has a rather desirable efficiency property not shared by other known constructions, which is that the allocation cells are stochastically ``small''. More precisely, let $X$ be the random diameter of the (almost surely unique) cell containing the origin. That is,
\begin{equation*}
X:=\diam(\psi_{\CZ}^{-1}(\psi_{\CZ}(0))).
\end{equation*}
Then we showed that for all $R\ge 1$ the inequality
\begin{equation}\label{eq:grav1thm}
\P(X>R) \le C \exp\bigg(-c R (\log R)^{\alpha_d}\bigg)
\end{equation}
holds, where $C,c$ are some positive constants that depend on $d$, and $\alpha_d = (d-2)/d$ if $d>3$ or can be taken to be any number less than $-4/3$ if $d=3$ (in which case $C,c$ will also depend on $\alpha_d$). In other words, the tail decay of the random diameter of the cell containing the origin is (at least) slightly faster than exponential in the distance in dimensions $4$ and higher, and (at least) almost exponential in dimension $3$.

\begin{figure}[h!]
\begin{center}

\subfigure[]{\raisebox{24pt}{\includegraphics[width=0.49\textwidth]{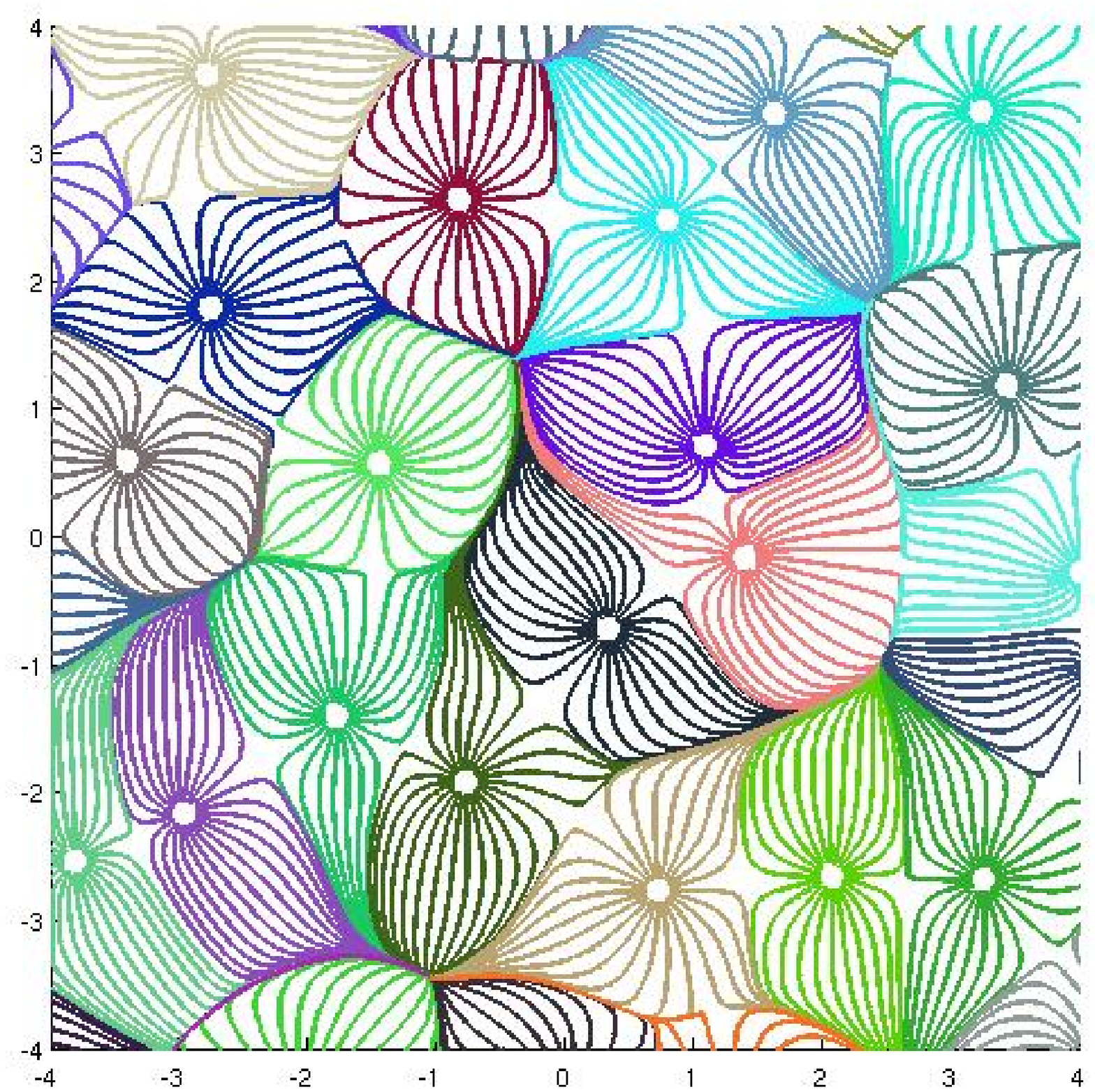}}}
\subfigure[]{\includegraphics[width=0.49\textwidth]{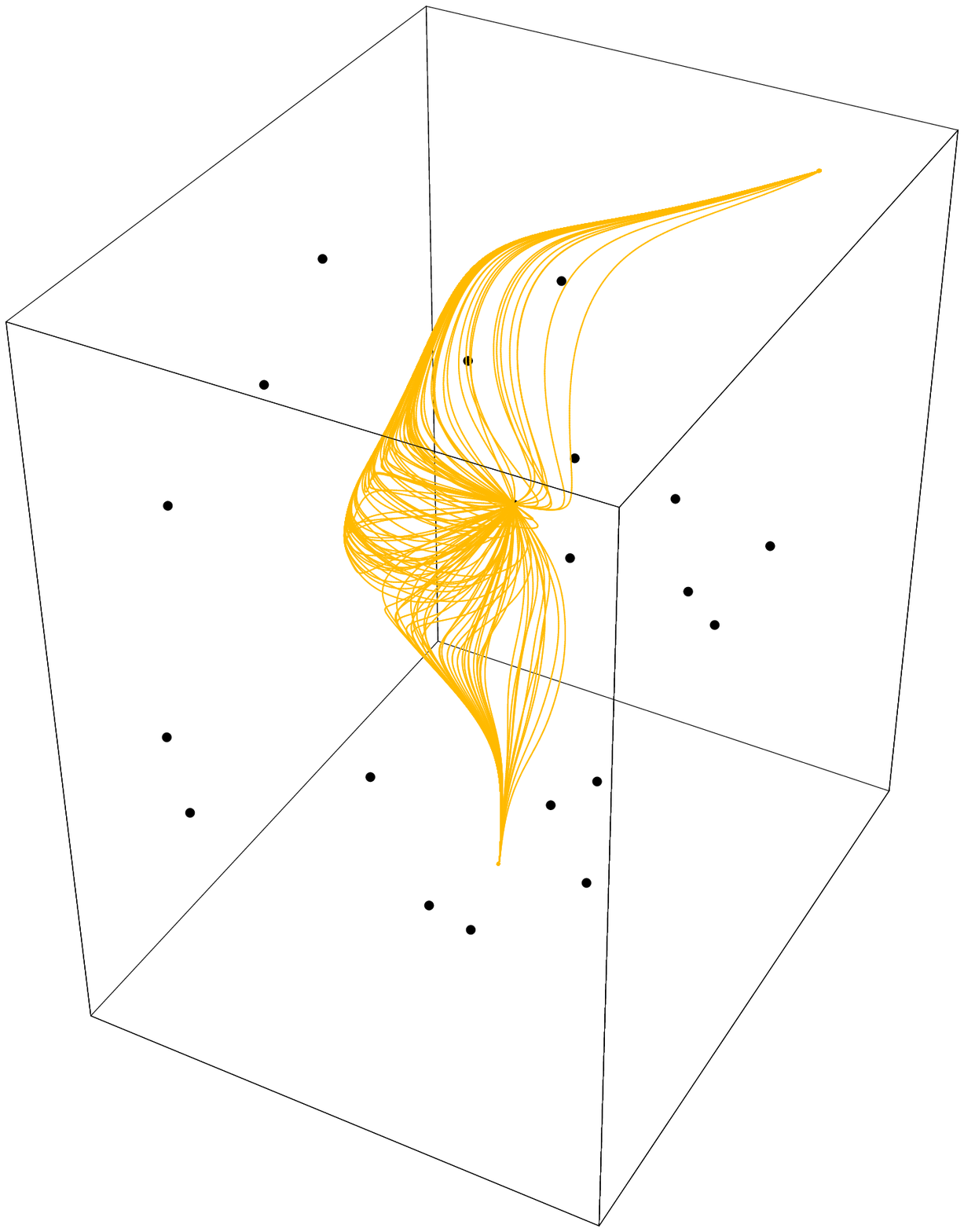}}

\caption{(a) The gradient flow allocation \cite{NSV07} (picture by Manjunath Krishnapur); (b) The gravitational allocation.\label{fig-model_sim}}
\end{center}
\end{figure}

One natural question is whether the bound in \eqref{eq:grav1thm} is sharp. We answer this question affirmatively (up to the lower order correction terms), and prove the following result.
\begin{theorem}\label{diameter_thm}
 For all dimensions $d\ge 3$ we have
 \begin{equation*}
  \P(X>R)=\exp(-R^{1+o(1)})
 \end{equation*}
 as $R\to\infty$.
\end{theorem}
The proof of the lower bound in Theorem~\ref{diameter_thm} is based on a precise understanding of the structure of the cells. Examining the structure in simulations, we see they have two parts with qualitatively different behavior: A massive central core, which is hard to move, and relatively ``thin'' tentacles, which are more flexible. This heuristic picture is captured by our main result, Theorem~\ref{main_theorem}, which pins down the spectrum of large deviation probabilities for the cell. The {\bf $R$-core} of a cell $B(z)$ is defined as the set $B(z)\cap B(z,R)$ (where $B(z,R)$ is the Euclidean ball of radius $R$ around $z$). The rest of the cell is termed the {\bf $R$-tentacles} of the cell.
\begin{theorem}\label{main_theorem}
Let
\begin{equation*}
Z_R:=\vol(\psi_{\cal Z}^{-1}(\psi_{\cal Z}(0))\setminus B(\psi_{\cal Z}(0),R)),
\end{equation*}
the volume of the $R$-tentacles of the cell containing the origin, and let
\begin{equation*}
 \begin{split}
  f_3(\gamma)&=\begin{cases}3-2\gamma&0\le\gamma\le1\\1&\gamma\ge 1\end{cases},\\ f_4(\gamma)&=\begin{cases}2-\frac{\gamma}{2}&0\le\gamma\le\frac{4}{3}\\4-2\gamma&\frac{4}{3}\le\gamma\le \frac{3}{2}\\1&\gamma\ge\frac{3}{2}\end{cases},\\
  f_d(\gamma)&=\begin{cases}1+\frac{2-\gamma}{d-2}&0\le\gamma\le2\\1&\gamma\ge 2\end{cases} \qquad (d\ge 5).
  \end{split}
 \end{equation*}
 Then for all dimensions $d\ge3$ and for all $\gamma>0$
 \begin{equation*}
 \P(Z_R>\exp(-R^\gamma))=\exp(-R^{f_d(\gamma)+o(1)})
 \end{equation*}
 as $R\to\infty$. Furthermore, there exists a $C>0$ such that for all $d\ge 3$
 \begin{equation*}
 \P(Z_R>R^{-C})=\exp(-R^{f_d(0)+o(1)})
 \end{equation*}
 as $R\to\infty$.
\end{theorem}
Figure \ref{fig-phasetrans} shows the functions $f_3$, $f_4$ and (schematically) $f_d$ for $d\ge5$.
\begin{figure}[h!]
\begin{center}
\resizebox{120pt}{!}{\includegraphics{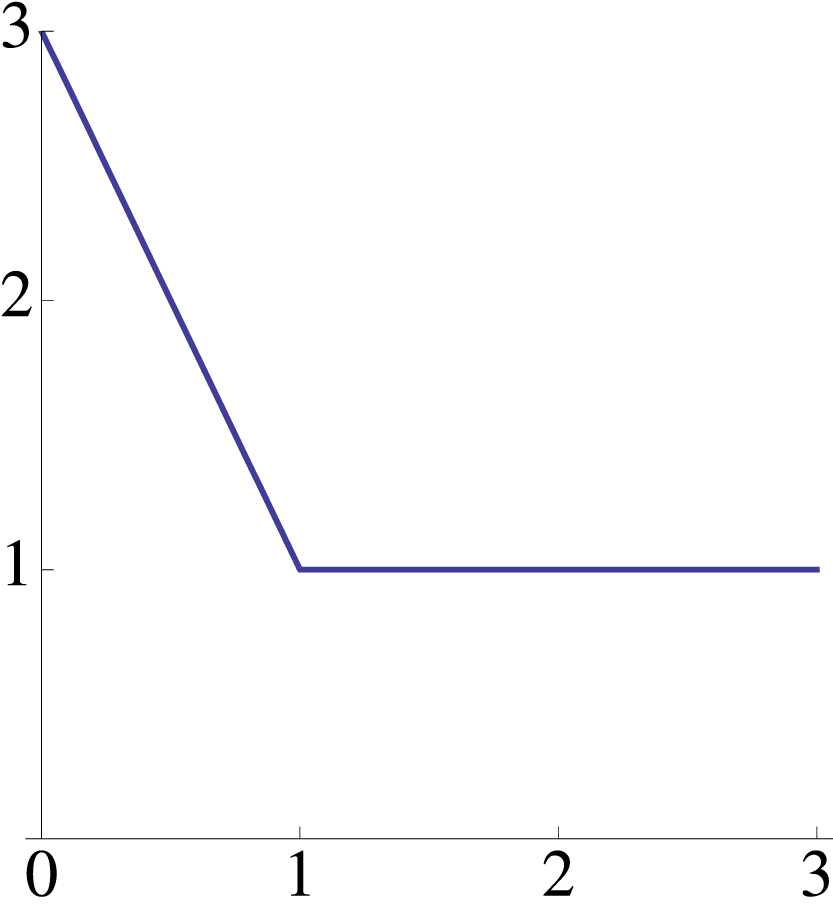}}\ \ 
\resizebox{120pt}{!}{\includegraphics{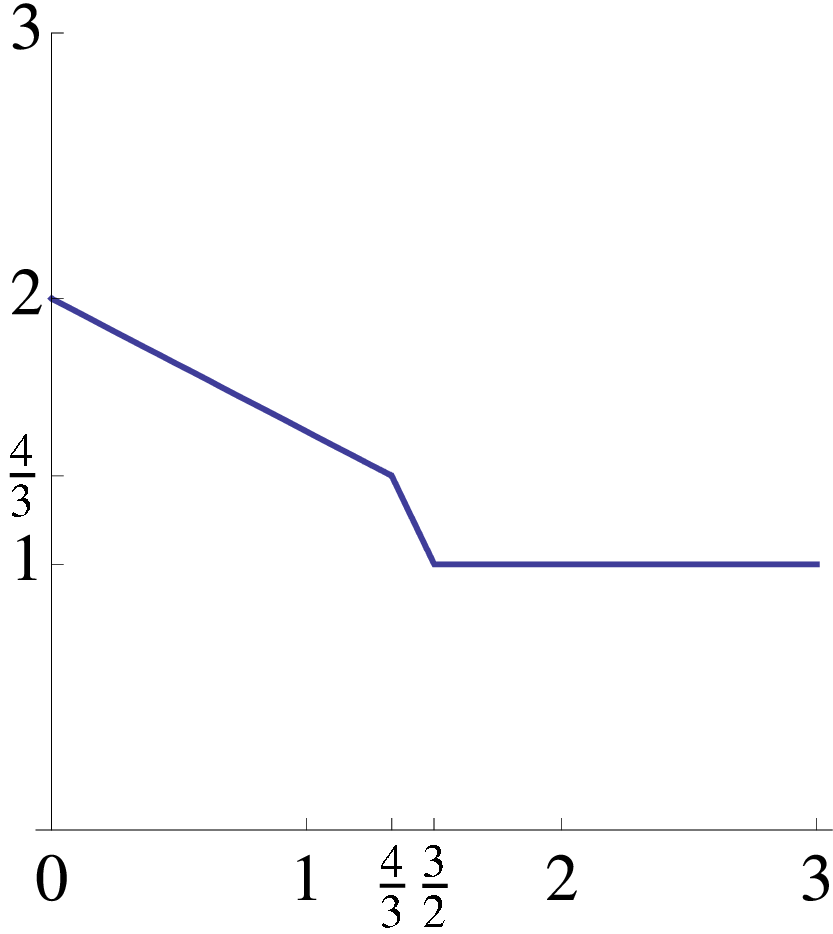}}\ \ 
\resizebox{120pt}{!}{\includegraphics{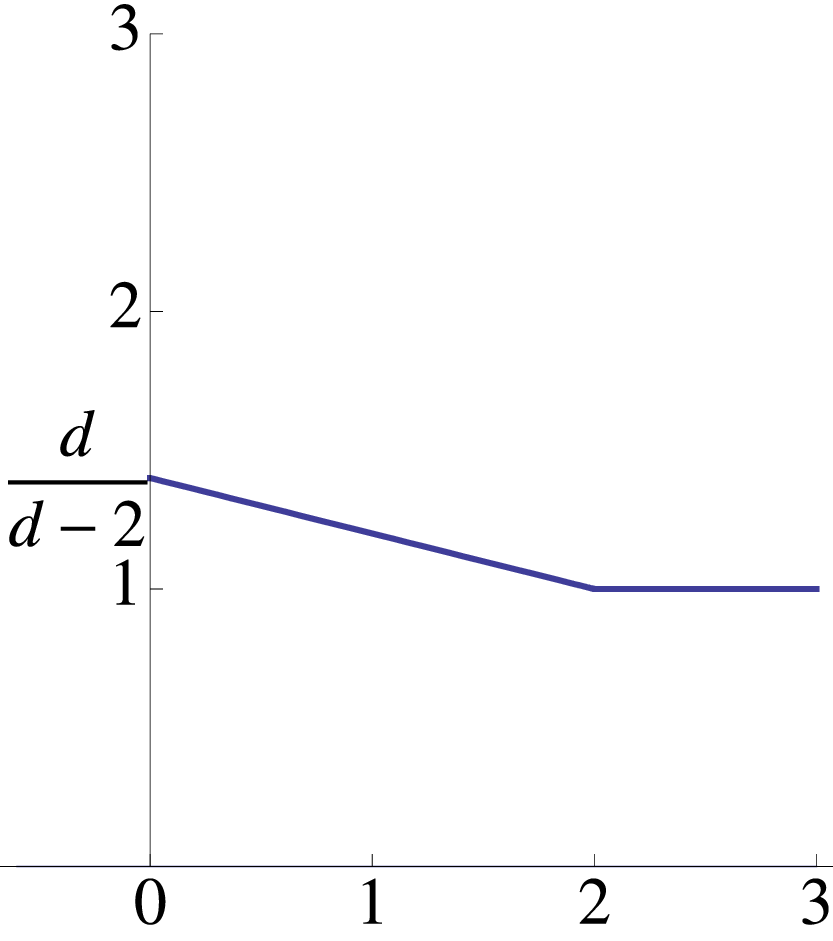}}
\caption{The functions $f_3$, $f_4$ and $f_d$ ($d\ge 5$).\label{fig-phasetrans}}
\end{center}
\end{figure}
The points of non-smoothness in the rate functions are classical signatures of phase transitions. The above theorem has several surprising features, in particular the double phase transition in dimension 4, which we now motivate.

One leading scenario causing mass $\exp(-R^\gamma)$ in a cell to travel distance $R$ is the existence of an {\bf attracting galaxy} (see figure~\ref{fig-cylinder-first}). More precisely, the ``attracting galaxy'' represents a region $U$ of volume $R^d$ having $cR^{d-\gamma}$ stars beyond its expectation; this event has probability $\exp(-R^{d-2\gamma+o(1)})$ for $0\le \gamma\le\frac{d}{2}$. However, we also need to control the stars in a channel of length $R$ and constant cross section in order for the mass to reach the distant attracting galaxy; obtaining this control (e.g., by keeping the channel empty of stars) has probability $\exp(-R^{1+o(1)})$. Taking both of these into account yields the expression for $f_3(\gamma)$.
\begin{center}
\begin{figure}[h!]
\hspace{5.0pt} \resizebox{370pt}{!}{\includegraphics{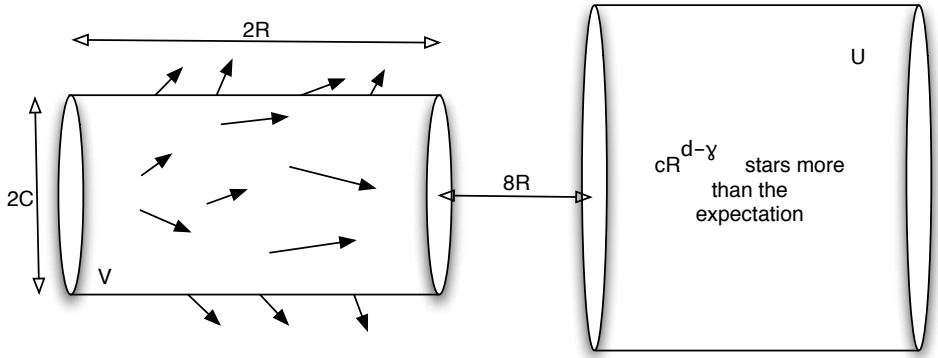}}
{\it
\caption{An attracting galaxy: Requiring $U$ of volume $R^d$ to have $cR^{d-\gamma}$ stars more than expected causes the required pull to the right.\label{fig-cylinder-first}}
}
\end{figure}
\end{center}

\begin{center}
\begin{figure}[h!]
\hspace{5.0pt} \resizebox{370pt}{!}{\includegraphics{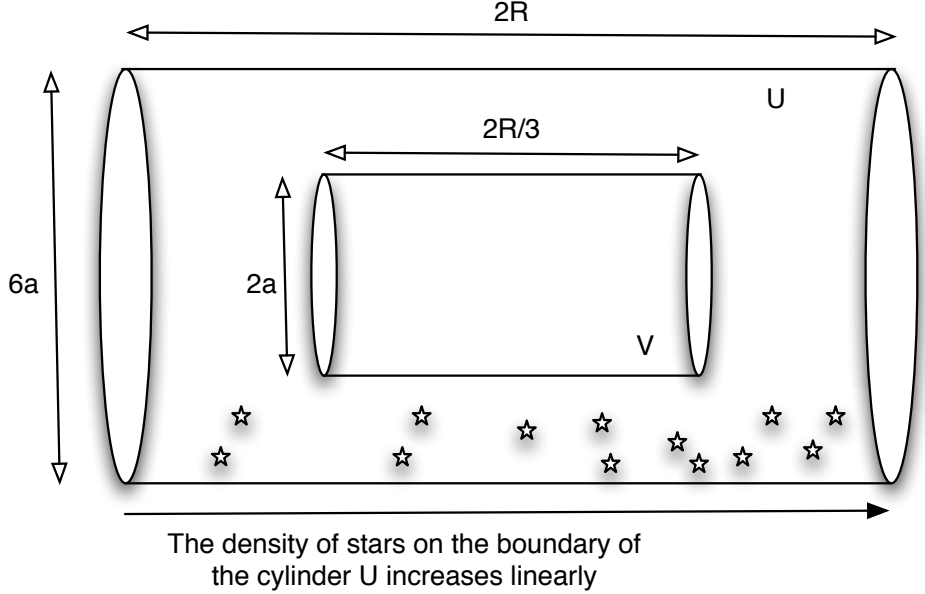}}
{\it
\caption{A wormhole. The radius $a$ is $R^{-(2-\gamma)/(d-2)+o(1)}$. Here the required pull to the right in $V$ is caused by the increasing density of stars on the curved part of the boundary of the cylinder $U$.
\label{fig-cylinder-second}}
}
\end{figure}
\end{center}
In high dimensions ($d\ge 5$) another scenario emerges as the dominant reason for mass $\exp(-R^{\gamma})$ in a cell to travel distance $R$: The existence of a {\bf wormhole}, a thin tube of radius $R^{-\frac{2-\gamma}{d-2}+o(1)}$ surrounded by $R^{1+\frac{2-\gamma}{d-2}+o(1)}$ stars arranged in rings of increasing density which ``pull'' mass through the tube (see figure~\ref{fig-cylinder-second}). This has probability $\exp(-R^{1+\frac{2-\gamma}{d-2}+o(1)})$. Fine control of stars within bounded distance of the wormhole is still needed; this has probability $\exp(-R^{1+o(1)})$ so we obtain the expression for $f_d(\gamma)$, $d\ge 5$.

In dimension 4, a wormhole is the dominant scenario when $\gamma<\frac{4}{3}$, but for $\gamma>\frac{4}{3}$ it is still cheaper to move mass using an attracting galaxy. 

A key challenge in proving the lower bounds is approximating smooth mass distributions using carefully placed discrete stars. This is based on the theory of Chebyshev-type cubatures which we apply in section~\ref{Chebyshev-type_cubature_section}. In fact, for our applications some new results in the theory of cubatures were needed; these are developed in \cite{P09}.

Theorem~\ref{main_theorem} reveals more about the geometry of the cells. Let $Y$ be the distance from a uniformly chosen point in the cell of the origin to the star of that cell. By translation equivariance, $Y$ may be written as
$$ Y = |\psi_\CZ(0)|, $$
the distance between the origin and the star of its cell. Clearly, $Y\le X$, but it turns out that in dimensions $d\ge 4$, it has much lighter tails.

\begin{theorem} \label{thm-gd}
Let $g_3=1$ and $g_d=1+\frac{1}{d-1}$ for $d\ge 4$. For all dimensions $d\ge3$ we have
\begin{equation*}
 \P(Y>R)=\exp(-R^{g_d+o(1)})
\end{equation*}
as $R\to\infty$.
\end{theorem}
The exponent $g_d$ is the unique $\gamma$ satisfying $f_d(\gamma)=\gamma$ for the function $f_d$ of Theorem \ref{main_theorem}. Surprisingly, in dimension 4 it coincides with the location of the first phase transition.

The case where a constant fraction of the cell's volume lies in its $R$-tentacles is also quite interesting. A simple lower bound for it is given in the next theorem. In dimension 3 this bound captures the correct exponent.
\begin{theorem}\label{simple_lower_bound_thm}
(a) For all $d\ge 3$ and $0<a<1$ there exist $C(a),c(a)>0$ such that if $R\ge C(a)$ we have
\begin{equation*}
\P(Z_R>a)\ge C(a)\exp(-c(a)R^d).
\end{equation*}
(b) For $d=3$ we have $\P(Z_R>a)=\exp(-R^{3+o(1)})$ as $R\to\infty$ for a fixed $0<a<1$.
\end{theorem}

\end{subsection}
\begin{subsection}{Sketch of the proofs}
In this section we sketch the proof of Theorem~\ref{main_theorem}. We prove separately the upper and lower bounds for $\P(Z_R>\exp(-R^\gamma))$. The upper bound for $\P(Z_R>R^{-C})$ follows from the other cases and the lower bound is proved similarly.

\subsubsection*{Lower bounds}

We start with the lower bounds (Section~\ref{lower_bound_section}). To prove the bound we explicitly construct an event with large enough probability on which the event $Z_R>\exp(-R^{\gamma})$ holds. We do this as follows. First, consider
the event that a ``cylinder'' of the form $V=[-R,R]\times (r S^{d-2})$
centered at the origin, with side length $R$ and radius $r$, has the
following properties (see Figure \ref{fig-cylinder1}):
\begin{enumerate}
\item[(I)] The cylinder $V$ contains no stars and has at each point a force whose first component is between $cR^{1-\gamma}$ and $CR^{1-\gamma}$ for some fixed $C,c>0$.
\item[(II)] The force at each point of the cylinder $V$'s boundary except, perhaps, for the ``caps'' $\{ \pm 1 \}\times (r S^{d-2})$ has an outward-pointing normal component.
\end{enumerate}

\begin{center}
\begin{figure}[h!]
\hspace{20.0pt} \includegraphics{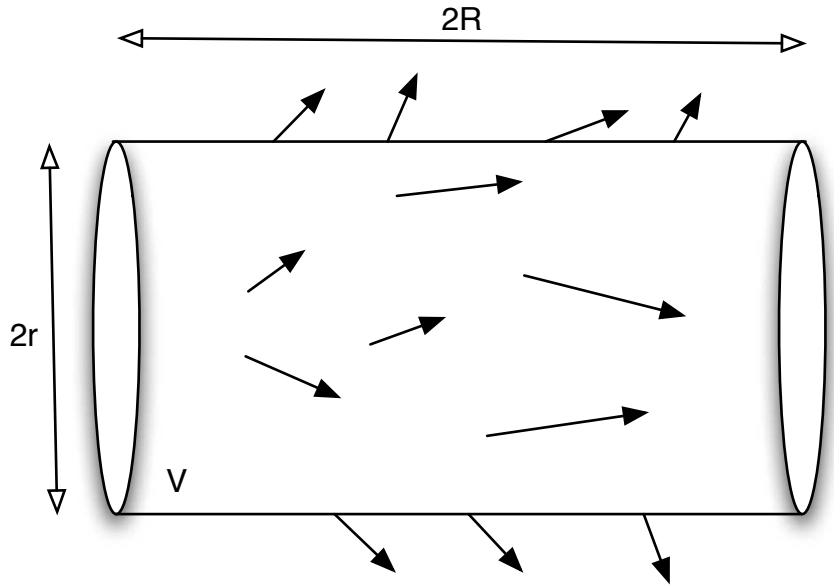}
{\it
\caption{An outline of the constructions for the lower bound. Arrows represent the gravitational force: a positive first-coordinate component (of order $R^{1-\gamma}$) inside the cylinder $V$ and an outward-pointing normal component on the curved part of the boundary.\label{fig-cylinder1}}
}
\end{figure}
\end{center}

By considering the backward flow of the force and using Liouville's theorem (equation \eqref{Liouvilles_thm_eq}) we deduce that if this event (intersected with another, highly probable, event) holds then there is a star close to the origin whose cell has more than $\exp(-\tilde{C}R^{\gamma})$ volume outside its $\tilde{c}R$-core. This implies the required lower bound. The rest of the lower bounds' proof consists of constructing an explicit event having the largest possible probability (in the exponential scale) on which the
conditions (I), (II) hold. We remark that a version of the above construction was also implicitly present in \cite{NSV07}.

To construct this event we place stars at certain roughly specified locations. Under this placement the expectation of the force satisfies (I) and (II). We then still need to prove that the force fluctuations induced by the rest of the stars do not change this expected picture; this will be explained further below. To place the stars, we use the more economical of two constructions according to the regime of the parameters $d$ and $\gamma$. The first construction, the attracting galaxy, is used for $d=3$, $0\le \gamma\le 1$ and for $d=4$, $\frac{4}{3}\le \gamma\le \frac{3}{2}$. The second construction, the wormhole, will give the lower bound for $d=4$, $0\le \gamma\le \frac{4}{3}$ and for $d=5$, $0\le \gamma\le 2$. The constructions differ in whether the pull in $V$ is due to ``far away'' or ``nearby'' stars. We now sketch these constructions.

{\bf Attracting galaxy:} In this construction we take $V$ to be of length $R$ and constant radius. We first require that $V$ should be empty of stars which automatically ensures that (II) is true for the expected force in $V$. We then consider a cylinder $U$ with dimensions of order $R$ located $10R$ units right of the origin and require that this cylinder contains order $R^{d-\gamma}$ stars more than its expectation (see figure~\ref{fig-cylinder-first}). These extra stars create the required estimate (I) (hence the name ``attracting galaxy''). The probabilistic cost of this construction is dominated by placing these extra stars and is $\exp(-CR^{d-2\gamma})$.

{\bf Wormhole:} The second construction is more complicated and is done only for $d\ge 4$. Here we take $V$ to be of length $\frac{1}{3}R$ and radius $R^{-\frac{2-\gamma}{d-2}+o(1)}$. We also consider $U:=3V$ and require that $U$ should be empty of stars. As before, this ensures that (II) holds for the expected force in $V$. We now place stars very close to the boundary (excluding the caps) of $U$ in a way which approximates a continuous density of stars (see figure~\ref{fig-cylinder-second}). More precisely, we place the stars so that for any $x\in V$, the force $\sum_z \frac{z-x}{|z-x|^d}$ due to these stars approximately equals $\int \frac{z-x}{|z-x|^d} d\nu(z)$ for a given measure $\nu$. The measure $\nu$ we use is the one supported on the boundary of $U$, excluding the caps, which is absolutely continuous with respect to the $(d-1)$-dimensional surface area measure and whose density depends only on the first coordinate, rising linearly from $R^{(2-\gamma)\frac{d-1}{d-2}+o(1)}$ at the left end to twice that at the right end. This placement of stars causes the expected force in $V$ to satisfy (I) (for $d\ge 4$) and leaves the estimate (II) intact. The probabilistic cost of this construction is dominated by the placement of stars approximating $\nu$ and equals $\exp(-R^{1+\frac{2-\gamma}{d-2}+o(1)})$ (which is approximately $\exp(-\nu(\R^d))$).

The main difficulty in the wormhole construction lies in the approximation of $\nu$ by stars. We require a very precise approximation and rely on a special type of Chebyshev-type cubature (Section~\ref{Chebyshev-type_cubature_section}). To this end, we divide (most of) the boundary of $U$ into pieces of small diameter $R^{-\frac{2-\gamma}{d-2}-o(1)}$ and equal measure $R^{o(1)}$ and for each piece we place stars at positions $(z_i)$ near the piece in a way that the discrete measure $\sum \delta_{z_i}$ has approximately the same first multi-moments as the measure $\nu$ restricted to that piece. By considering the Taylor expansion of the force (Section~\ref{Taylor_expansion_section}) and using the smallness of the diameter we observe that such an approximation suffices to approximate the force in $V$.

{\bf Controlling the fluctuations:} As mentioned above, these constructions only cause the expected force in $V$ to satisfy estimates (I) and (II), we also need to show that the fluctuations induced by all the stars whose locations were not specified do not significantly affect this expected force. It turns out that the main contribution to the force fluctuations comes from stars at distances between $R^{o(1)}$ and $R$ from the set $V$ (those more distant typically induce small fluctuations as shown by moderate deviation estimates and we require that closer stars do not exist). These fluctuations turn out to typically be too large and to overcome this we prove a small ball estimate lower bounding the probability that they are all small. Theorem~\ref{small_ball_estimate_thm} (roughly) says that the fluctuations to the expected force in $V$ from the stars at distances between $R^{o(1)}$ and $R$ are smaller than $R^{-d}$ with probability at least $\exp(-R^{1+o(1)})$. This theorem is one of the main and difficult components in our proof and a sketch of it is given in Section~\ref{small_ball_estimate_section}. It again relies on a special type of Chebyshev-type cubatures, this time showing that the stars lie on an approximate cubature with a lower bounded probability.
\subsubsection*{Upper bounds}
The proof of the upper bounds (Section~\ref{upper_bound_proof_sec}) relies on ideas from \cite{NSV07} but requires a more complicated analysis due to the stronger fluctuations of the Poisson process.

{\bf Dimensions 5 and higher:} Our starting point is an observation coming directly from Liouville's theorem which says that in each cell, the volume of the set of points taking time at least $t$ to travel to the star is exactly $\exp(-d\kappa_d t)$, where $\kappa_d = \pi^{d/2}/\Gamma(d/2+1)$ is the volume of the unit ball in $\R^d$. Letting $z$ be the star of the cell of the origin, this already implies that if, say, $Z_{4R}>\exp(-d\kappa_d R^\gamma)$ then there is a point $x$ in the cell of the origin taking time less than $R^{\gamma}$ to travel to $z$ and satisfying $|x-z|>4R$. Next,
recalling that the potential $U$ (see \eqref{grav_potential}) decreases along flow curves,
 we consider the flow curve of the point $x$ and divide it into three parts. The part from $x$ to the first point $x_1$ on which $U(x_1)=R^{2-\gamma}$, the part from $x_1$ till the first point $x_2$ where $U(x_2)=-R^{2-\gamma}$ and the part from $x_2$ to $z$. The next observation is that a gravitational flow curve cannot travel far if it it flows for a short time with a small potential change. Since $x$ travels to $z$ in less than $R^{\gamma}$ time we deduce from this that the path from $x_1$ to $x_2$ has diameter smaller than $\sqrt{2R}$. But recalling that $|x-z|>4R$, we see that either the first or the third part of the path must have diameter at least $R$. Summarizing the above, we have shown for $d\ge 5$ that if $Z_{4R}>\exp(-d\kappa_d R^{\gamma})$ then there is a curve in the cell of the origin whose diameter is at least $R$ and on which $|U(x)|>R^{2-\gamma}$. Using \eqref{eq:grav1thm} to estimate the diameter of the cell we may also assume that this curve is not too far from the origin.

{\bf Long curves with atypical potential:} The main theorem in the upper bounds section then says that the probability of a curve as above is at most $\exp(-R^{1+(2-\gamma)/(d-2)+o(1)})$. The reason behind this is that the main contribution to the probability of $|U(x)|>R^{2-\gamma}$ comes from having some star at distance $cR^{-(2-\gamma)/(d-2)}$ from $x$. So if this was the only way $|U(x)|$ would be large then we would have to have at least $R^{1+(2-\gamma)/(d-2)}$ disjoint (and hence independent) occurrences of this which would yield the required bound. The main difficulty is in showing that indeed, having $|U|$ large along the curve because of many stars further away, although it affects $|U|$ at more points, is still less likely than having the effect come mainly from nearby stars. This is achieved using a multi-scale analysis in which we partition space into finitely many slabs $(A_i)$ and discretize distance to finitely many scales $(L_i)$ and then for each possibility of assigning a scale $L_{j(i)}$ to a slab $A_i$ we estimate the probability that there exists a point in $A_i$ having large potential due to the effect of stars at distance of order $L_{j(i)}$ (for the smallest scale we estimate the probability of many points in $A_i$ to be affected by this scale).

{\bf Dimensions 3 and 4:} The above approach needs to be slightly modified for dimensions $3$ and $4$ since the stationary potential $U$ does not exist. Instead we work with the potential difference function $\pd$ which should be thought of as $U(x)-U(y)$ for two points $x,y$.  Most of the ideas and techniques from dimensions 5 and higher carry over to this case; however, one main difference is that in some regime of the parameters (namely, when $d=3$ or when $d=4$ and $4/3 \le \gamma \le 3/2$) the main contribution to the potential difference $\pd(y,x)$ is from stars which are ``far away'' (formally: at distance at least of order $R^{(2-\gamma)/2}$) from $x$ or $y$. We bound this contribution using the large deviation theorems developed in \cite{CPPR07} and find that it is significant with probability at most $\exp(-cR^{3-2\gamma})$ for $d=3$ and at most $\exp(-cR^{4-2\gamma})$ for $d=4$. In the regime described above this probability dominates the estimate (this explains the appearance of the first phase transition in dimension 4).
\end{subsection}
\begin{subsection}{Proofs of Theorems~\ref{thm-gd} and \ref{simple_lower_bound_thm}}

\begin{proof}[Proof of Theorem \ref{thm-gd}]
Fix $R>0$ and let $d\mu$ denote the distribution of $Z_R$. Then since $Y$ is the distance between the star of the cell of 0 and a uniformly chosen point in its cell we have that
 \begin{eqnarray*}
  \P(Y>R) &=& \int_0^1 vd\mu(v) = \int_0^1 \P(Z_R>v)dv \\ &=& \int_0^{v_R} \P(Z_R>v)dv+\int_{v_R}^1 \P(Z_R>v)dv
 \end{eqnarray*}
 for any $0<v_R<1$. Since we also have
 \begin{eqnarray*}
  v_R\P(Z_R>v_R)&\le& \int_0^{v_R} \P(Z_R>v)dv+\int_{v_R}^1 \P(Z_R>v)dv\\ &\le& v_R+(1-v_R)\P(Z_R>v_R)
 \end{eqnarray*}
 the result follows by choosing $v_R=\exp(-R^{g_d})$ and using Theorem \ref{main_theorem}.
\end{proof}

\begin{proof}[Proof of Theorem~\ref{simple_lower_bound_thm}]
Part (b) follows from part (a) and Theorem~\ref{main_theorem}. For part (a), fix $R>0$, $0<a<1$ and let $N_R$ be the number of stars in $B(0,R)$. Let 
\begin{equation*}
\begin{split}
E_1&:=\{N_R> \frac{1}{1-a}\vol(B(0,2R))\},\\
E_2&:=\{\text{There is no gravitational flow curve connecting $\partial B(0,R^{2d})$}\\
&\qquad \text{and }\partial B(0,2R^{2d})\},\\
E_x&:=\{\text{The cell containing $x$ has at least $a$ volume in its $R$-tentacles}\}.
\end{split}
\end{equation*}
On the event $E_1$ we note that one of the stars in $B(0,R)$ must have at least $a$ volume in its $R$-tentacles. Then if both $E_1$ and $E_2$ occurred then the cell of that star is contained in $B(0,2R^{2d})$. Next, denoting $B:=B(0,2R^{2d})$ and using Fubini's theorem we have
\begin{equation*}
\vol(B)\P(Z_R>a)=\int_B\P(E_x)dx=\E\int_B 1_{E_x}dx\ge \E1_{E_1\cap E_2}\int_B 1_{E_x}dx\ge \P(E_1\cap E_2).
\end{equation*}
The proof is completed by noting that $\P(E_1)\ge \frac{c(a)}{R^{d/2}}\exp(-C(a)R^d)$ (see also Lemma~\ref{poissonlemma}) and using \cite[Theorem 3]{CPPR07}, one of the main results of \cite{CPPR07}, to obtain $\P(E_2^c)\le C\exp(-cR^{2d}/\log^C R)$.
\end{proof}
\end{subsection}
\end{section}

\begin{section}{Notation and background}\label{notation_section}
In this paper we use $C$ and $c$ for positive real constants which depend only on $d$ unless explicitly stated otherwise. We may change the values of $C$ and $c$ from line to line; $C$ may be increased and $c$ may be decreased.

We let $\vol$ stand for Lebesgue measure, $\sigma_{m}$ for the
$m$-dimensional area measure on sets in $\R^d$ and $|\cdot|$ for the Euclidean norm.

Throughout the proof of the lower bounds we will make use of boxes and cylinders centered around the origin. Our boxes and cylinders will be parallel to the axes and the boxes will have equal dimensions in the $x_2,\ldots, x_d$ directions. Hence we define
\begin{equation*}
\begin{split}
\Box{L}{W}&:=\{x\in\R^d\ |\ |x_1|\le L,\ |x_i|\le W\text{ for all $2\le i\le d$}\}\\
\Cyl{L}{W}&:=\{x\in\R^d\ |\ |x_1|\le L,\ x_2^2+x_3^2+\cdots+x_d^2\le W^2\}
\end{split}
\end{equation*}
For a cylinder $U:=\Cyl{L}{W}$ we will write $\partial' U:=\{x\in\partial U\ |\ |x_1|<L\}$. That is, the boundary of $U$ excluding the ``caps'' of the cylinder.

For a vector field $G:\R^d\to\R^d$, let $G(x)_i$ for $1\le i\le d$ be the $i$'th component of $G(x)$. Let $G(x)_n$ be the \emph{cylindrical} radial component of $G(x)$, i.e., 
\begin{equation*}
G(x)_n:=G(x)\cdot \frac{(0,x_2,x_3,\ldots, x_d)}{|(0,x_2,x_3,\ldots, x_d)|}.
\end{equation*}
Similarly let $x_i$ be the $i$'th coordinate of $x$.

Recall from \cite{CPPR07} that $F(x | A)$ stands for the gravitational force at
$x$ as exerted by the stars in a set $A\subset \R^d$ and normalized to have mean $0$.
More precisely, for a bounded set $A$ it is defined by
\begin{equation}\label{force_on_bounded_set}
F(x | A) := \sum_{z\in\CZ\cap A} \frac{z-x}{|z-x|^{d}}
  - \int_A \frac{z-x}{|z-x|^d} d\vol(z),
  \end{equation}
and for a set $A$ whose complement is bounded it is defined by $F(x
| A):=F(x)-F(x | A^c)$. Similarly, the gravitational potential at $x$ from stars in $A$ is defined by
$$ U(x | A) := \frac{1}{d-2}\sum_{z\in\CZ\cap A} \frac{-1}{|z-x|^{d-2}} + \frac{1}{d-2}\int_A |z-x|^{d-2} d\vol(z) $$
for a bounded set $A$, and, for dimension $d\ge 5$, by
$$ U(x | A) = U(x) - U(x | A^c) $$
for a set whose complement is bounded, where
\begin{equation}\label{grav_potential}
U(x) = \frac{1}{d-2}\lim_{T\to\infty} \Big[ \sum_{z\in\CZ\cap B(0,T)} \frac{-1}{|z-x|^{d-2}} + \frac{d\kappa_d}{2} T^2
\Big] - \frac{\kappa_d}{2}|x|^2
\end{equation}
is the total gravitational potential; see \cite[Section 7]{CPPR07} ($U(x)$ converges only for $d\ge 5$. For $d=3,4$ we define the potential difference function, see Section~\ref{pot_diff_large_dev_sec}).

Next, we define $g:\R^d\to\R^d$ by
\begin{equation}\label{def_of_g}
g(z):=\frac{z}{|z|^d}.
\end{equation}
We will make use of the facts that
\begin{equation}\label{g_and_Dg_estimates}
|g(z)|=|z|^{1-d} \qquad\text{and}\qquad |D_1 g(z)|\le C|z|^{-d}
\end{equation}
where $D_1$ stands for the first differential. The second fact is shown in \cite[Eq. (10)]{CPPR07} and is also a corollary of Theorem \ref{Taylor_expansion_thm} in this paper. We let $\alpha\in (\N\cup\{0\})^d$ stand for a multi-index. We write $|\alpha| := \sum_{i=1}^d \alpha_i$  and $x^\alpha:=\prod_{i=1}^d x_i^{\alpha_i}$ for $x\in\R^d$. For any $k\ge 1$ we let
\begin{equation*}
\polydim(k,d):=\binom{k+d}{d}-1
\end{equation*}
and define the moment map $P_k^d:\R^d\to\R^{\polydim(k,d)}$ by
\begin{equation}\label{moment_map_def}
P_k^d(x):=(x^{\alpha})_\alpha
\end{equation}
where the index runs over all multi-indices with $0<|\alpha|\le k$.

Finally, for a given set $D\subseteq \R^d$, we let $\diam(D):=\sup_{x,y\in D}|x-y|$ stand for the diameter of $D$.

We will make use of some deviation inequalities for a Poisson random variable. The following lemma is standard:
\begin{lemma} \label{poissonlemma}
Let $X$ be a Poisson random variable with mean $\lambda>0$. Then:\\
(i) If $t\ge 2\lambda$ then $$\P(X\ge t)\le e^{-\frac{1}{4}t\log\left(\frac{t}{\lambda}\right)}. $$
(ii) There exists a $\delta>0$ such that for all $t\in[0,\delta\lambda]$ we have
$$ \P(|X-\lambda|\ge t) \le 2e^{-t^2/3\lambda}. $$
(iii) There exists $c>0$ such that if $n\ge\lambda$ is an integer then
\begin{equation*}
\P(X=n)\ge \frac{c}{\sqrt{n}}\exp(-(n-\lambda)^2/\lambda)
\end{equation*}
\end{lemma}
\begin{proof}
Parts (i) and (ii) are proven, for example, in \cite[Lemma 4]{CPPR07}. For part (iii), note that Stirling's approximation gives that for $n\ge 1$, $n!\le C\sqrt{n}\left(\frac{n}{e}\right)^n$. Hence
\begin{equation*}
\P(X=n)=e^{-\lambda}\frac{\lambda^n}{n!} \ge \frac{c}{\sqrt{n}}\exp(n-\lambda-n\log(n/\lambda)).
\end{equation*}
And using the fact that $\log(n/\lambda)=\log(1+\frac{n-\lambda}{\lambda})\le \frac{n-\lambda}{\lambda}$ we obtain
\begin{equation*}
\P(X=n)\ge\frac{c}{\sqrt{n}}\exp\left(n-\lambda-\frac{n(n-\lambda)}{\lambda}\right)=\frac{c}{\sqrt{n}}\exp\left(-\frac{(n-\lambda)^2}{\lambda}\right).\qedhere
\end{equation*}
\end{proof}

We will use a simple consequence of a version of Liouville's theorem \cite[p. 69, Lemma 1]{A89} (see also \cite[Section 4]{CPPR07}).
\begin{lemma}\label{Liouville_thm_lemma}
Let $A\subset\R^d$ be a measurable set and let $A_t$ be its image under the gravitational flow after $t$ time units. Then if no point of $A$ has reached a star during the evolution then we have
\begin{equation}\label{Liouvilles_thm_eq}
\vol(A_t)=e^{d\kappa_d t}\vol(A).
\end{equation}
\end{lemma}








\end{section}
\begin{section}{Deviation estimates}
\begin{subsection}{Large deviation estimates}
In \cite{CPPR07}, large deviation estimates were proven for the potential, force and derivative of the force. More precisely, in Theorem 17 and Corollary 18, given $\infty\ge p>q>0$, the quantities $\max_{x\in B(0,1\wedge\frac{q}{2})}|U(x|B(0,p)\setminus B(0,q))|$, $\max_{x\in B(0,1\wedge\frac{q}{2})}|F(x|B(0,p)\setminus B(0,q))|$ and $\max_{x\in B(0,1\wedge\frac{q}{2})}|D_1F(x|B(0,p)\setminus B(0,q))|$ were considered and large deviation estimates for the their right tail were derived. In this section we assert that these same large deviation estimates hold also when the potential, force or force derivative are restricted to a general domain instead of a difference of two balls. In this new setting, the role of $q$ is played by the closest point to the origin in the domain.

\begin{theorem} \label{large_deviation_thm}
There exist constants $C_1, c_2, c_3>0$ such that for any measurable set $A$ which is either bounded or has bounded complement, letting $q:=\min_{x\in A} |x|>0$ and $t>0$ we have
\begin{eqnarray}
\p\Big(\max_{x\in B(0,1\wedge \frac{q}{2})} \Big|U\big(x\ \big|\ A\big)\Big|\ge t \Big) &\le&
C_1 e^{-c_2 q^{d-2} t \log\left(\frac{c_3 t}{q^2}\right)}, \label{eq:ballfirstone} \\
\p\Big(\max_{x\in B(0,1\wedge \frac{q}{2})}\Big|F\big(x\ \big|\ A\big)\Big|\ge t \Big) &\le&
C_1 e^{-c_2 q^{d-1} t \log\left(\frac{c_3 t}{q}\right)}, \label{eq:ballsecondone} \\
\p\Big(\max_{x\in B(0,1\wedge \frac{q}{2})} \Big|D_1 F\big(x\ \big|\ A\big)\Big|\ge t \Big) &\le&
C_1 e^{-c_2 q^d t \log\left(c_3 t \right)}, \label{eq:ballthirdone}
\end{eqnarray}
where equation \eqref{eq:ballfirstone} holds in dimensions $d\ge 5$, and equations \eqref{eq:ballsecondone} and \eqref{eq:ballthirdone} hold for all dimensions $d\ge 3$.
\end{theorem}
\begin{proof}
This theorem is analogous to Corollary 18 in \cite{CPPR07} with the set $A$ replacing the set $\R^d\setminus B(0,q)$ which appeared there. To prove it, one proves an analogue of Theorems 16 and 17 of \cite{CPPR07} and deduces the current theorem as a corollary, as is done there. The proofs of these analogues are exactly the same as the original proofs in \cite{CPPR07}, with a few notational changes. Since these changes are minor, we omit the full proofs and merely detail the changes.

In Theorem 16, $B_{p,q}$ is replaced by $B(0,p)\cap A^c$, $W_{p,q}$ is replaced by a uniform random point in $B(0,p)\cap A^c$, $N_{p,q}$ is replaced by the number of stars in $B(0,p)\cap A^c$ and $U_{p,q}$ is replaced by the sum $\sum_{z_i\in B(0,p)\cap A^c} \frac{1}{|z_i|^{d-2}}$. In theorem 17, all references to $B(0,p)\setminus B(0,q)$ are replaced by $B(0,p)\cap A^c$, and all references to $B(0,p)\cap (B(0,2^{m+1}q)\setminus B(0,2^m q))$ are replaced by $B(0,p)\cap (B(0,2^{m+1}q)\setminus B(0,2^m q)) \cap A^c$.
\end{proof}
\end{subsection}

\begin{subsection}{Moderate deviation estimates}
In this section, moderate deviation estimates will be derived for the force and potential. It is possible to prove such estimates also for the derivative of the force but we shall not need this.

\begin{theorem}\label{moderate_deviation_at_point_thm}
There exist constants $C_1, c_2, c_3>0$ such that for any measurable set $A$ which is either bounded or has bounded complement, if $q:=\min_{x\in A} |x|>0$ and $t>0$ then
\begin{eqnarray}
\p\Big(\Big|U\big(0\ \big|\ A\big)\Big|\ge t \Big) &\le&
C_1 e^{-c_2 q^{d-4} t^2}, \label{eq:pointfirstonemoderate} \\
\p\Big(\Big|F\big(0\ \big|\ A\big)\Big|\ge t \Big) &\le&
C_1 e^{-c_2 q^{d-2} t^2}, \label{eq:pointsecondonemoderate}
\end{eqnarray}
where equation \eqref{eq:pointfirstonemoderate} holds in dimensions $d\ge 5$ for $t\le c_3q^2$, and equation \eqref{eq:pointsecondonemoderate} holds in dimensions $d\ge 3$ for $t\le c_3 q$.
\end{theorem}

\begin{proof}
We shall prove \eqref{eq:pointsecondonemoderate}; the proof of \eqref{eq:pointfirstonemoderate} is similar and is omitted.
Define $A_m := A\cap(B(0,8^{m+1} q)\setminus B(0,8^m q))$ for integer $m\ge 0$. Set $t_m:=2^{-(m+1)}t$. Note that since a.s. $F(0\ |\ A) = \sum_{m=0}^\infty F(0\ |\ A_m)$, it is enough to prove that for some $c_4>0$
\begin{equation} \label{moderate_deviation_cylinder_point_Am}
\begin{split}
P(|F(0\ |\ A_m)| \ge t_m) &\le C\exp(-c_4 (8^mq)^{d-2}t_m^2)=\\
&=C\exp(-c_4 2^{3m(d-2)-2(m+1)}q^{d-2}t^2)
\end{split}
\end{equation}
since if $c_4 q^{d-2}t^2\ge 1$ then equation \eqref{eq:pointsecondonemoderate} follows from equation \eqref{moderate_deviation_cylinder_point_Am} by a union bound, and if $c_4 q^{d-2}t^2 < 1$, then equation \eqref{eq:pointsecondonemoderate} can be made true just by choosing the constant $C_1$ large enough.


Let us now prove estimate \eqref{moderate_deviation_cylinder_point_Am}. Fix $m\ge 0$. We may assume $A_m\neq \emptyset$ since otherwise there is nothing to prove. Let $(Z_i)_{i=1}^\infty$ be an IID sequence of uniformly distributed
points in $A_m$ that are independent of all other random variables. Let $N\sim\Poisson(\vol(A_m))$ denote the number of stars in $A_m$ and note that given $N=n$, these stars are distributed as $Z_1, \ldots, Z_n$. Hence, recalling the definition of $g$ from \eqref{def_of_g}, we have
\begin{equation*}
\begin{split}
F(0\ |\ A_m) &\stackrel{d}{=} \sum_{i=1}^N g(Z_i) - \int_{A_m} g(z)dz =\\
&= \sum_{i=1}^N g(Z_i) - \vol(A_m)\E g(Z_1) = \\
&= \sum_{i=1}^N \left(g(Z_i) - \E g(Z_1)\right) + (N-\vol(A_m))\E g(Z_1).
\end{split}
\end{equation*} 
It follows that to prove \eqref{moderate_deviation_cylinder_point_Am}, it is enough to show
\begin{align}
\P(|\sum_{i=1}^N (g(Z_i) - \E g(Z_1))|\ge \frac{t_m}{2}) \le C\exp(-c(8^mq)^{d-2}t_m^2),\label{Hoeffding_Poisson_bound}\\
\P(|(N-\vol(A_m))\E g(Z_1)|\ge \frac{t_m}{2}) \le C\exp(-c(8^mq)^{d-2}t_m^2). \label{Poisson_moderate_deviations_exp_appl}
\end{align}
We start by noting that for $z\in A_m$ we have
\begin{equation} \label{A_m_g_bound}
|g(z)|\le (8^mq)^{-(d-1)}.
\end{equation}
To prove \eqref{Hoeffding_Poisson_bound} we use the Bernstein-Hoeffding inequality \cite{hoeffding} to obtain
\begin{equation} \label{Hoeffding_given_N}
\P(|\sum_{i=1}^N \left(g(Z_i) - \E g(Z_1)\right)|\ge \frac{t_m}{2}\ |\ N)\le C\exp(-c\frac{(8^mq)^{2d-2}t_m^2}{N}).
\end{equation}
Now by averaging on $N$ we deduce that for any $\rho>0$
\begin{equation*}
\P(|\sum_{i=1}^N \left(g(Z_i) - \E g(Z_1)\right)|\ge \frac{t_m}{2})\le C\exp(-c\frac{(8^mq)^{d-2}t_m^2}{\rho})+\P(N\ge \rho(8^mq)^d).
\end{equation*}
Hence, using the assumption that $t\le c_3q$, \eqref{Hoeffding_Poisson_bound} will be proven if we show that for large enough $\rho$,
\begin{equation*}
\P(N\ge \rho(8^mq)^d)\le \exp(-c\rho(8^mq)^d).
\end{equation*}
This latter estimate follows immediately from Lemma~\ref{poissonlemma} upon recalling that $N\sim\Poisson(\vol(A_m))$ and $\vol(A_m)\le C(8^mq)^d$.

It remains to prove estimate \eqref{Poisson_moderate_deviations_exp_appl}. In view of \eqref{A_m_g_bound}, it is enough to show that
\begin{equation*}
\P(|(N-\vol(A_m))|\ge \frac{(8^mq)^{d-1}t_m}{2})\le C\exp(-c(8^mq)^{d-2}t_m^2).
\end{equation*}
We divide into two cases:
\begin{enumerate}
\item If $\frac{(8^mq)^{d-1}t_m}{2}\ge 2\vol(A_m)$, we obtain from Lemma~\ref{poissonlemma} that
\begin{multline*}
\P(|(N-\vol(A_m))|\ge \frac{(8^mq)^{d-1}t_m}{2})\le \P(N\ge \frac{(8^mq)^{d-1}t_m)}{2}\le\\
\le \exp(-c(8^mq)^{d-1}t_m)\le \exp(-c(8^mq)^{d-2}t_m^2)
\end{multline*}
where in the last inequality we used the assumption that $t\le c_3 q$.
\item If $\frac{(8^mq)^{d-1}t_m}{2}<2\vol(A_m)$, we obtain from Lemma~\ref{poissonlemma}
\begin{multline*}
\P(|(N-\vol(A_m))|\ge \frac{(8^mq)^{d-1}t_m}{2})\le 2\exp\left(-c\frac{(8^mq)^{2d-2}t_m^2}{\vol(A_m)}\right)\le\\
\le 2\exp(-c(8^mq)^{d-2}t_m^2)
\end{multline*}
where in the last inequality we used the fact that $\vol(A_m)\le C(8^mq)^d$.
\end{enumerate}
\end{proof}

\begin{theorem}\label{moderate_deviation_in_ball_thm}
There exist constants $C_1, c_2, c_3>0$ such that for any measurable set $A$ which is either bounded or has bounded complement, letting $q:=\min_{x\in A} |x|>0$ and $t>0$ we have
\begin{eqnarray}
\p\Big(\max_{x\in B(0,1)} \Big|U\big(x\ \big|\ A\big)\Big|\ge t \Big) &\le&
C_1 (1+t^{-d})e^{-c_2 q^{d-4} t^2}, \label{eq:ballfirstonemoderate} \\
\p\Big(\max_{x\in B(0,1)}\Big|F\big(x\ \big|\ A\big)\Big|\ge t \Big) &\le&
C_1 (1+t^{-d})e^{-c_2 q^{d-2} t^2}, \label{eq:ballsecondonemoderate}
\end{eqnarray}
where equation \eqref{eq:ballfirstonemoderate} holds in dimensions $d\ge 5$ for $t\le c_3q^2$, and equation \eqref{eq:ballsecondonemoderate} holds in dimensions $d\ge 3$ for $t\le c_3 q$.
\end{theorem}

\begin{proof}
Note that we may assume $q>2$ since the estimates hold trivially when $q\le 2$ by the assumptions on $t$. We prove \eqref{eq:ballsecondonemoderate}; the proof of \eqref{eq:ballfirstonemoderate} is similar and is omitted. We wish to use Theorems~\ref{large_deviation_thm} and \ref{moderate_deviation_at_point_thm}. There are two cases to consider; denote $\eta=4/c_3$ where in this appearance only, $c_3$ is the constant appearing in equation \eqref{eq:ballthirdone}, then:
\begin{enumerate}
\item If $t\ge \eta$, we obtain
\begin{multline*}
\P(\max_{x\in B(0,1)}|F(x\ |\ A)|\ge t) \le\\
\le \P(|F(0\ |\ A)|\ge t/2)+\P(\max_{x\in B(0,1)} |D_1 F(x\ |\ A)|\ge t/2) \le \\
\le C\exp(-cq^{d-2}t^2) + C\exp(-cq^d t\log(2\eta^{-1}t)) \le \\
\le C\exp(-cq^{d-2}t^2) + C\exp(-cq^d) \le C\exp(-cq^{d-2}t^2)
\end{multline*}
where the last inequality follows by the theorem's assumption that $t\le c_3 q$. This proves the theorem for this case.
\item If $t<\eta$: Cover the ball $B(0,1)$ by $K$ balls of radius $0<r<1$ (to be specified later) with centers in $B(0,1)$. This is possible with $K\le Cr^{-d}$ balls. Let $x_1,x_2,\ldots,x_K$ be the centers of these balls. A union bound gives
\begin{multline*}
\P(\max_{x\in B(0,1)}|F(x\ |\ A)|\ge t) \le\\
\le \sum_{i=1}^K \left(\P(|F(x_i\ |\ A)|\ge t/2)+\P(\max_{y\in B(x_i,r)} |D_1 F(y\ |\ A)|\ge t/2r)\right) \le \\
\le K\left(C\exp(-cq^{d-2}t^2) + C\exp(-cq^d\frac{t}{r}\log(\frac{2\eta^{-1}t}{r}))\right)
\end{multline*}
We now choose $r=\eta^{-1}t$ (which is indeed smaller than $1$ since $t<\eta$) so that $K\le Ct^{-d}$ and obtain from the previous inequality and the assumption that $t\le c_3 q$
\begin{multline*}
\P(\max_{x\in B(0,1)}|F(x\ |\ A)|\ge t) \le\\
\le  Ct^{-d}(\exp(-cq^{d-2}t^2) + \exp(-cq^d\log(2))) \le\\
\le Ct^{-d}\exp(-cq^{d-2}t^2)
\end{multline*}
which proves the theorem for this case.\qedhere
\end{enumerate}
\end{proof}

In most of our uses, the set $A$ of the previous theorem will be of the form $B(0,p)\setminus B(0,q)$. We now prove an extension of this theorem to ``moving annuli''. This will be convenient in the bounds of Section~\ref{upper_bound_proof_sec}.
\begin{theorem}\label{moderate_deviation_in_moving_ball_thm}
There exist constants $C_1, c_2, c_3,c>0$ such that for all $p>q>0$ and $t>0$ we have that
\begin{align}
\p\Big(\max_{x\in B(0,1)} \Big|U\big(x\ \big|\ B(x,p)\setminus B(x,q)\big)\Big|\ge t \Big) &\le
C_1 (1+\frac{a^d}{t^{d}})(1+t^{-d})e^{-c_2 q^{d-4} t^2}, \label{eq:ballfirstonemoderate_moving} \\
\p\Big(\max_{x\in B(0,1)}\Big|F\big(x\ \big|\ B(x,p)\setminus B(x,q)\big)\Big|\ge t \Big) &\le
C_1 (1+t^{-d})^2e^{-c_2 q^{d-2} t^2}, \label{eq:ballsecondonemoderate_moving}
\end{align}
where $a=p$ if $p<\infty$ and $a=q$ if $p=\infty$ and where equation \eqref{eq:ballfirstonemoderate_moving} holds in dimensions $d\ge 5$ for $t\le c_3q^2$, and equation \eqref{eq:ballsecondonemoderate_moving} holds in dimensions $d\ge 3$ for $t\le c_3 q$.
\end{theorem}
\begin{proof}
 We prove \eqref{eq:ballfirstonemoderate_moving}; \eqref{eq:ballsecondonemoderate_moving} is proven similarly and its proof is omitted. First, we prove \eqref{eq:ballfirstonemoderate_moving} in the limiting case when $p=\infty$. Let $\eta>0$ be a small constant and fix
$x\in B(0,\eta(1\wedge\frac{t}{q}))$. Then
$$
\Big|U\big(x\ \big|\ \R^d \setminus B(x,q)\big) -
U\big(x\ \big|\ \R^d \setminus B(0,q)\big)\Big|
\qquad\qquad\qquad\qquad $$

\vspace{-18.0pt}
\begin{eqnarray*} \qquad\qquad &=&
\Big|U\big(x\ \big|\ B(0,q)\big) -
U\big(x\ \big|\ B(x,q)\big)\Big| \\ &=&
-\kappa_d |x|^2/2 - \frac{1}{d-2}\sum_{z_i\in E_1}
\frac{1}{|z_i-x|^{d-2}} + \frac{1}{d-2}\sum_{z_i\in E_2}
\frac{1}{|z_i-x|^{d-2}},
\end{eqnarray*}
where $E_1 = B(0,q)\setminus B(x,q)$ and $E_2 = B(x,q)\setminus
B(0,q)$. Now, denoting by $\nu_q$ the number of stars in
$B\left(0,q+\eta(1\wedge\frac{t}{q})\right)\setminus B\left(0,q-\eta(1\wedge\frac{t}{q})\right)$, it follows
that
$$
\Big|U\big(x\ \big|\ \R^d \setminus B(x,q)\big)\Big| \le
\Big|U\big(x\ \big|\ \R^d \setminus B(0,q)\big)\Big| + C\eta \frac{t^2}{q^2} +
\frac{\nu_q}{(d-2)(q/2)^{d-2}}. $$
Since $\nu_q$ is a Poisson random
variable with mean $\le C\eta q^{d-2}t$, by Lemma \ref{poissonlemma} we obtain that for
$t\le cq^2$ and small enough $\eta$ we have
\begin{eqnarray}
\p\left( \max_{x\in B(0,\eta(1\wedge\frac{t}{q}))}
\Big|U\big(x\ \big|\ \R^d \setminus B(x,q)\big)\Big| \ge t \right)
&\le&\nonumber \\
\p\left( \max_{x\in B(0,\eta(1\wedge\frac{t}{q}))}
\Big|U\big(x\ \big|\ \R^d \setminus B(0,q)\big)\Big| \ge t/3 \right)
& + &
\p\left( \frac{\nu_q}{(d-2)(q/2)^{d-2}}\ge t/3 \right) \le\nonumber \\ \le 
C (1+t^{-d})e^{-c q^{d-4} t^2}+e^{-cq^{d-2}t}&\le& C(1+t^{-d})e^{-cq^{d-4}t^2}.
\label{eq:haveproved}
\end{eqnarray}
We fix $\eta>0$ so that this estimate holds. We now cover the ball $B(0,1)$ with fewer than $C(1+\frac{q^{d}}{t^d})$ balls of radius $\eta(1\wedge\frac{t}{q})$ and use the above estimate for each such ball to obtain 
\begin{equation*}
\p\left( \max_{x\in B(0,1)}
\Big|U\big(x\ \big|\ \R^d \setminus B(x,q)\big)\Big| \ge t \right)\le C(1+\frac{q^d}{t^{d}})(1+t^{-d})e^{-cq^{d-4}t^2}
\end{equation*}
as required for the case $p=\infty$. Note that we have assumed that $q\ge 1$ since the above estimate holds trivially if $q<1$ since $t\le cq^2$.

Finally, to prove \eqref{eq:ballfirstonemoderate_moving} in the general case, note, using \eqref{eq:haveproved} twice
and using the assumption $t\le cq^2\le cp^2$, that
\begin{equation*}
\begin{split}
&\p\Big(\max_{x\in B(0,1)} \Big|U\big(x\ \big|\ B(x,p) \setminus B(x,q)\big)\Big|\ge t \Big)\le \\
&\p\Big(\max_{x\in B(0,1)} \Big|U\big(x\ \big|\ \R^d \setminus B(x,q)\big)\Big|\ge \frac{t}{2} \Big)+
\p\Big(\max_{x\in B(0,1)} \Big|U\big(x\ \big|\ \R^d \setminus B(x,p)\big)\Big|\ge \frac{t}{2} \Big)\\
&\le C(1+\frac{q^d}{t^{d}})(1+t^{-d})e^{-c q^{d-4} t^2} + C(1+\frac{p^d}{t^{d}})(1+t^{-d})e^{-c p^{d-4} t^2}\le\\
&\le C(1+\frac{p^d}{t^{d}})(1+t^{-d})e^{-c q^{d-4} t^2}.\qedhere
\end{split}
\end{equation*}
\end{proof}

\end{subsection}
\begin{subsection}{Large deviations for the potential difference function}\label{pot_diff_large_dev_sec}
Recall from \cite{CPPR07} that in dimensions 3 and 4, the stationary potential function $U$ does not exist and we must content ourselves with the potential difference function $\pd(x,y)$. $\pd$ exists in all dimensions $d\ge 3$ and when $d\ge 5$ we have $\pd(x,y)=U(y)-U(x)$. 
In dimensions 3 and 4, recall that the potential difference function is defined by
$$ \pd(x,y) = \frac{1}{d-2} \sum_{z\in \CZ, |z|\uparrow}\left( \frac{-1}{|z-y|^{d-2}}-\frac{-1}{|z-x|^{d-2}}\right)
+\frac{\kappa_d}{2}(|x|^2-|y|^2). $$
If $A\subset \R^d$ is a bounded set, define
\begin{eqnarray*}
\pd(x,y | A) &=& \frac{1}{d-2} \sum_{z\in \CZ\cap A, |z|\uparrow}\left( \frac{-1}{|z-y|^{d-2}}-\frac{-1}{|z-x|^{d-2}}\right) \\ & &
- \frac{1}{d-2}\int_A \left( |z-x|^{d-2}-|z-y|^{d-2} \right) d\vol(z),
\end{eqnarray*}
and if the complement of $A$ is bounded define
$$ \pd(x,y | A) = \pd(x,y) - \pd(x,y | A). $$

We now state a slight extension of \cite[Corollary 34]{CPPR07} that will be important in the proof of the upper bound in the next section.
We note that although it was stated in \cite{CPPR07} for finite and infinite $p$, it was in fact proved (and used) there only for the case $p=\infty$. This is also the case which we will need here. The range of $t$ for this theorem extends slightly beyond what was stated in \cite{CPPR07}. 
\begin{theorem} \label{large_dev_pot_diff_thm}
In dimension $d=4$, there exist constants $C_1, c_2, c_3>0$ such that
for all $x,y\in \R^d$ and $q>2$ satisfying $|x-y|>3q$, we have that
\begin{multline*}
 \p\left(  \max_{u\in B(x,1), v\in B(y,1)} \Big| \pd\big(u,v\ \big|\ \R^d \setminus
(B(u,q)\cup B(v,q)) \big)\Big| > t \right) \\ \le C_1 e^{-c_2 q^2 t \log\left(\frac{c_3 t}{q^2}\right)}
\end{multline*}
for all $t$ satisfying $t\ge C_1 q^2\text{ and }t\ge C_1 q^2 \log\left(\frac{C_1 t}{q^2}\right)\log\left(\frac{|x-y|}{q}\right)$.
  
Similarly, in dimension $d=3$,there exist constants $C_1, c_2, c_3>0$ such that
for all $x,y\in \R^d$, $q>2$ and $t>0$ satisfying $|x-y|>3q$, $c_3t>q$ and $t < |x-y|q$
 we have that
\begin{multline*}
 \p\left(  \max_{u\in B(x,1), v\in B(y,1)} \Big| \pd\big(u,v\ \big|\ \R^d \setminus
(B(u,q)\cup B(v,q)) \big)\Big| > t \right) \\
\le C_1 e^{-\frac{c_2 t^2}{|x-y|}} + C_1 e^{-c_2 q^2
t\log\left(\frac{c_3 t}{q}\right)}.
\end{multline*}
\end{theorem}
As written above, this theorem differs from that stated in \cite{CPPR07} in that it is valid for the case $p=\infty$ and also in that the upper restriction on $t$ when $d=3$ has been extended from $t<|x-y|$ to all $t<|x-y|q$. Since the proof given in \cite{CPPR07} works verbatim for this extension as well, we will not repeat it here.
\end{subsection}
\end{section}

\begin{section}{Proof of the main theorem - upper bound}\label{upper_bound_proof_sec}

In this section we prove the upper bound for Theorem~\ref{main_theorem}. I.e., we show that $\P(Z_R>\exp(-R^{\gamma}))\le\exp(-R^{f_d(\gamma)+o(1)})$ for $\gamma>0$ and the functions $f_d$ given in the theorem. Note that the upper bound for the case $\P(Z_R>R^{-C})$ follows from the other cases. Also note that the cases where $f_d(\gamma)=1$ follow from the main theorem of \cite{CPPR07}, hence we shall prove the bound only for the remaining cases. The proof relies on ideas from \cite{NSV07} but requires a more complicated multi-scale analysis due to the stronger fluctuations of the Poisson process.

The theorem is a consequence of Theorem~\ref{atypical_potential_on_curve_thm} below and the following two simple lemmas. The first lemma relates the time it takes the gravitational flow to pass a certain distance and the potential change along that flow. A similar lemma appeared in \cite{NSV07} with a more complicated proof.
\begin{lemma}\label{time_and_potential_rel_lemma}
Let $\Gamma:[0,t_*]\to\R^d$ be (a segment of) a gravitational flow curve and let $L$ measure the arc length along $\Gamma$. Then for $d\ge 5$ we have
\begin{equation*}
L(t_*)^2 \le t_*(U(\Gamma(0))-U(\Gamma(t_*))).
\end{equation*}
and for $d=3$ or $d=4$ we have for any $x\in\R^d$
\begin{equation*}
L(t_*)^2 \le t_*(\pd(x,\Gamma(0))-\pd(x,\Gamma(t_*))).
\end{equation*}
\end{lemma}
\begin{proof}
For $d\ge 5$ we have
\begin{eqnarray*}
t_* &=& \int_\Gamma dt = \int_\Gamma \frac{dt}{dL} dL = \int_\Gamma \frac{1}{dL/dt} dL
= \int_\Gamma\frac{1}{|F|}dL \\
&=& L(t_*)\int_\Gamma\frac{1}{|F|} \frac{dL}{L(t_*)} \substack{\text{\ \ (by convexity)\ \ } \\ \ge}
L(t_*) \left( \int_\Gamma|F| \frac{dL}{L(t_*)}\right)^{-1} \\
&=& L(t_*)^2 \left( \int_\Gamma |F|dL\right)^{-1} = \frac{L(t_*)^2}{U(\Gamma(0))-U(\Gamma(t_*))}.
\end{eqnarray*}
This calculation works also for dimensions 3 and 4 by recalling that for any $x\in\R^d$, $\nabla_y\pd(x,y)=-F(y)$, hence
\begin{equation*}
\int_\Gamma |F|dL = \pd(x,\Gamma(0))-\pd(x,\Gamma(t_*)).\qedhere
\end{equation*}
\end{proof}
Our second lemma is a direct consequence of Liouville's theorem.
\begin{lemma}\label{exponential_flow_time_lemma}
 Consider the set of points in the cell of the origin taking time at least $t$ to flow into the star under the gravitational flow. Let $V_t$ be the volume of this set. Then for any $d\ge 3$ we have $V_t=\exp(-d\kappa_d t)$.
\end{lemma}
\begin{proof}
This was proved in \cite[Section 4]{CPPR07} before equation (7) and also follows from Lemma~\ref{Liouville_thm_lemma} along with the fact that the cell has volume 1.
\end{proof}

The main theorem of this section says that it is very unlikely to have a long curve with atypically large (positive or negative) potential throughout. We give separate statements for $d=3,4$ and for $d\ge 5$ since the potential function $U$ does not exist for $d=3,4$. More precisely, let us define for $R,\delta,\rho>0$
\begin{eqnarray*}
 \Omega_{R,\delta} &:=& \Big\{ \exists\text{ continuous path }\Gamma\subseteq B(0,R^{2d}) \text{ of diameter at least }R\\
 & & \qquad \text{ such that }|U(x)|\ge R^\delta\text{ for all }x\in\Gamma\Big\},\\
 \Omega'_{R,\delta,\rho} &:=& \Big\{ \exists\text{ continuous path }\Gamma\subseteq B(0,R^{2d}) \text{ of diameter at least }R\\
 & & \qquad \text{ such that }|U(x|B(x,\rho))|\ge R^\delta\text{ for all }x\in\Gamma\Big\}.
\end{eqnarray*}
\begin{theorem}\label{atypical_potential_on_curve_thm}
Let $h_d(\delta)=1+\frac{\delta}{d-2}$. Then for any $d\ge 5, \delta>0, R>1$ and $0<\eps<\frac{\delta}{2(d-2)}$ there exist $C(\eps,\delta),c(\eps,\delta)>0$ such that
\begin{equation*}
\P(\Omega_{R,\delta})\le C(\eps,\delta) \exp\left(-c(\eps,\delta) R^{h_d(\delta)-\eps}\right).
\end{equation*}
Furthermore, for any $d\ge 3, \delta>0, R>1$ and $0<\eps<\frac{\delta}{2(d-2)}$ there exist $C(\eps,\delta),c(\eps,\delta)>0,c>0$ such that for any $1\le\rho\le cR^{\delta/2}$ we have
\begin{equation*}
\P(\Omega'_{R,\delta,\rho})\le C(\eps,\delta) \exp\left(-c(\eps,\delta) R^{h_d(\delta)-\eps}\right).
\end{equation*}
\end{theorem}

We first show how the upper bounds follow from this theorem and the lemmas above, and then proceed to prove Theorem \ref{atypical_potential_on_curve_thm}.

\begin{proof}[Proof of the upper bound in  Theorem~\ref{main_theorem}]
We divide into two cases.

{\bf Dimension $d\ge 5:$} Fix $d\ge 5$, $0<\gamma<2$ and let
\begin{equation*}
 \Omega:=\{Z_{4R}>\exp(-d\kappa_d R^\gamma)\}.
\end{equation*}
Let $T_R$ be the $R$-tentacles for the cell of the origin, i.e., $T_R:=\psi_{\CZ}^{-1}(\psi_{\CZ}(0))\setminus B(\psi_{\CZ}(0),R)$. By Lemma~\ref{exponential_flow_time_lemma}, if $\Omega$ occurred then there is a $x\in T_{4R}$ with $\tau_x\le R^{\gamma}$ where $\tau_x$ denotes the travel time of $x$ to the star $\psi_\CZ(0)$. Consider $\Gamma:[0,\tau_x]\to\R^d$, $\Gamma(0)=x$, the gravitational flow curve of $x$, and define
\begin{equation*}
\begin{split}
\tau_1&:=\min(\tau\ |\ U(\Gamma(\tau))\le R^{2-\gamma}),\\
\tau_2&:=\min(\tau\ |\ U(\Gamma(\tau))\le -R^{2-\gamma}).
\end{split}
\end{equation*}
Noting that $\tau_2-\tau_1\le \tau_x\le R^{\gamma}$ and applying Lemma~\ref{time_and_potential_rel_lemma} we deduce that $|\Gamma(\tau_2)-\Gamma(\tau_1)|\le \sqrt{2}R$. Hence since $x\in T_{4R}$ we finally deduce that either
\begin{equation*}
|\Gamma(0)-\Gamma(\tau_1)|\ge R\ \ \text{ or }\ \ |\Gamma(\tau_2)-\Gamma(\tau_x)|\ge R.
\end{equation*}
To summarize the above discussion, let
\begin{equation*}
\Omega_1:=\{\text{The diameter of the cell containing the origin is more than $R^{2d}$}\}
\end{equation*}
and note that on $\Omega_1^c$, the cell of the origin is contained in $B(0,R^{2d})$. Then we have shown that
\begin{equation*}
\Omega\subseteq \Omega_1\cup \Omega_{R,2-\gamma}.
\end{equation*}
Since by \eqref{eq:grav1thm}, the main theorem of \cite{CPPR07}, we have 
\begin{equation}\label{diameter_of_cell_estimate}
\P(\Omega_1)\le C\exp(-cR^{2d}\log^{\alpha_d}(R))
\end{equation}
it remains to apply Theorem~\ref{atypical_potential_on_curve_thm} and observe that the exponent functions satisfy $f_d(\gamma)=h_d(2-\gamma)$.

{\bf Dimensions 3 and 4:} Fix $d=3$ or $d=4$ and fix $\gamma>0$ satisfying $f_d(\gamma)> 1$. For $\eta>0$ let
\begin{align*}
\rho_3&:=R^{(2-\gamma)/2}\log^{-1/2-\eta}(R),\\
\rho_4&:=\eta R^{(2-\gamma)/2}.
\end{align*}
Define an event
\begin{equation*}
\begin{split}
\Omega_2:=\{&\exists x\in B(0,R^{2d})\text{ such that }\forall y\in B(0,R^{2d}) \text{ with }6R\le |x-y|\le 7R\\
 &\text{ we have }|U(y|B(y,\rho_d))|\ge R^{2-\gamma}\}.
\end{split}
\end{equation*}
We note that if $R>C$ then $\Omega_2\subseteq \Omega'_{R,2-\gamma,\rho_d}$. In particular, if $\eta<c$ then by Theorem~\ref{atypical_potential_on_curve_thm} we have for $\eps<\frac{2-\gamma}{4}$,
\begin{equation}\label{small_d_long_curve_atypical_potential_estimate}
\P(\Omega_2)\le C(\eps,\gamma) \exp\left(-c(\eps,\gamma) R^{h_d(2-\gamma)-\eps}\right).
\end{equation}
Define also the event
\begin{equation*}
\begin{split}
\Omega_3:=\{&\exists x,y\in B(0,R^{2d})\text{ with } 2R\le |x-y|\le 11R\text{ such that }\\
&|\pd(y,x|\R^d\setminus (B(x,\rho_d)\cup B(y,\rho_d))|\ge R^{2-\gamma}\}.
\end{split}
\end{equation*}
We note that by the large deviation theorem \ref{large_dev_pot_diff_thm} we have that if $0<\gamma<2$, $0<\eta<1$ and $R>C(\gamma,\eta)$ then
\begin{equation}\label{small_d_U_diff_estimate}
\P(\Omega_3)\le \begin{cases}C\exp\big[-c(\eta)R^{4-2\gamma}\log^{-1-2\eta}(R)\log\log(R)\big]&d=4\\C\exp(-cR^{3-2\gamma})&d=3\end{cases}.
\end{equation}
We fix $\eta$ so that \eqref{small_d_long_curve_atypical_potential_estimate} and \eqref{small_d_U_diff_estimate} hold. We now proceed as in the case $d\ge 5$ and let
\begin{equation*}
 \Omega_4:=\{Z_{4R}>\exp(-\frac{1}{2}d\kappa_d R^\gamma)\}.
\end{equation*}
Assume that $\Omega_1^c\cap\Omega_2^c\cap\Omega_3^c\cap\Omega_4$ occurred ($\Omega_1$ is as for the case $d\ge 5$). Let $T_R$ be the $R$-tentacles for the cell of the origin, i.e., $T_R:=\psi_{\CZ}^{-1}(\psi_{\CZ}(0))\setminus B(\psi_{\CZ}(0),R)$. By Lemma~\ref{exponential_flow_time_lemma}, since $\Omega_4$ occurred there is a $x\in T_{4R}$ with $\tau_x\le \frac{1}{2}R^{\gamma}$ where $\tau_x$ denotes the travel time of $x$ to the star $\psi_\CZ(0)$. Since $\Omega_2^c$ occurred there exists $y\in B(0,R^{2d})$ with $6R\le |x-y|\le 7R$ and 
\begin{equation}\label{typical_trunc_potential_at_y}
|U(y|B(y,\rho_d))|\le R^{2-\gamma}.
\end{equation}

Now consider $\Gamma:[0,\tau_x]\to\R^d$, $\Gamma(0)=x$, the gravitational flow curve of $x$, and define
\begin{equation*}
\begin{split}
\tau_1&:=\min(\tau\ |\ \pd(y,\Gamma(\tau))\le 3R^{2-\gamma}),\\
\tau_2&:=\min(\tau\ |\ \pd(y,\Gamma(\tau))\le -3R^{2-\gamma}).
\end{split}
\end{equation*}
Noting that $\tau_2-\tau_1\le \tau_x\le \frac{1}{2}R^{\gamma}$ and applying Lemma~\ref{time_and_potential_rel_lemma} we deduce that $|\Gamma(\tau_2)-\Gamma(\tau_1)|\le \sqrt{3}R$. Hence since $x\in T_{4R}$ we deduce that either
\begin{equation}\label{two_possible_atypical_potential_curves}
|\Gamma(0)-\Gamma(\tau_1)|\ge R\ \ \text{ or }\ \ |\Gamma(\tau_2)-\Gamma(\tau_x)|\ge R.
\end{equation}
Assume the former and let $\tau_1':=\min(\tau\ |\ |\Gamma(0)-\Gamma(\tau)|=R)$, then since $\Omega_3^c$ occurred we know that for any $w\in\Gamma([0,\tau_1'])$ we have
\begin{equation*}
|\pd(y,w|\R^d\setminus (B(w,\rho_d)\cup B(y,\rho_d))|< R^{2-\gamma}.
\end{equation*}
Combining this with \eqref{typical_trunc_potential_at_y} we finally obtain for every $w\in\Gamma([0,\tau_1'])$ that $U(w|B(w,\rho_d))\ge R^{2-\gamma}$, so that in particular $\Omega'_{R,2-\gamma,\rho_d}$ occurred. Similarly if the second option in \eqref{two_possible_atypical_potential_curves} occurred then we would also conclude that $\Omega'_{R,2-\gamma,\rho_d}$ occurred.

Summarizing the above discussion we have shown that
\begin{equation*}
\Omega_4\subseteq \Omega_1\cup\Omega_2\cup\Omega_3\cup \Omega'_{R,2-\gamma,\rho_d}.
\end{equation*}
Hence the required bound for $\Omega_4$ follows from \eqref{diameter_of_cell_estimate},\eqref{small_d_long_curve_atypical_potential_estimate},\eqref{small_d_U_diff_estimate} and Theorem~\ref{atypical_potential_on_curve_thm}.

\end{proof}

%
%
%
%
%
%


\subsection*{The probability of a long curve with atypical potential}
We now prove Theorem~\ref{atypical_potential_on_curve_thm}. 
Let $Q(x,L)=x+[-L,L]^d$ denote the cube centered at $x$ with side lengths $2L$.
Fix $\delta > 0$ and consider the event
\begin{align*}
E_{R,\delta} &:= \bigl\{\text{$\exists$ a continuous path $\Gamma$ from $\partial Q(0,R)$ to $\partial Q(0,2R)$} \\
&\qquad \qquad \text{such that $|U(x)|\ge R^\delta$ for every $x\in \Gamma$.}\bigr\}
\end{align*}
Also let
\begin{align*}
E'_{R,\delta,\rho}&:=\bigl\{\text{$\exists$ a continuous path $\Gamma$ from $\partial Q(0,R)$ to $\partial Q(0,2R)$}\\
&\qquad \qquad \text{such that $|U(x|B(x,\rho))|\ge R^{\delta}$ for every $x\in\Gamma$.}\bigr\}
\end{align*}
We will prove
\begin{theorem}
Suppose $d\ge 5$ and $\delta > 0$, and let $E_{R,\delta}$ be defined as above. Then for any $0 < \alpha < \frac{\delta}{2(d-2)}$, there exist $C(\delta,\alpha),c(\delta,\alpha)>0$ such that for all $R> 0$,
\[
\p(E_{R,\delta}) \le C(\delta,\alpha)\exp\bigl(-c(\delta,\alpha)R^{1+\frac{\delta}{d-2} - \alpha}\bigr).
\]
Moreover, for $d= 3$ or $4$, we have that for any $\delta>0$ and any $0<\alpha<\frac{\delta}{2(d-2)}$ there exist $C(\delta,\alpha),c(\delta,\alpha),c>0$ such that for all $R>0$ and $1\le \rho\le cR^{\delta/2}$, we have
\[
P(E'_{R,\delta,\rho})\le C(\delta,\alpha)\exp(-c(\delta,\alpha)R^{1+\frac{\delta}{d-2}-\alpha}). 
\]
\end{theorem}

It is straightforward to see by covering $B(0,R^{2d})$ by $R^c$ boxes $Q(0,cR)$ that $\Omega_{R,\delta}$ is contained in $R^C$ translates of $E_{cR,\delta}$ and that $\Omega'_{R,\delta,\rho}$ is contained in $R^C$ translates of $E'_{cR,\delta,\rho}$. Hence Theorem~\ref{atypical_potential_on_curve_thm} follows from the above theorem.

{\it Proof.}  In the following, $C$ and $c$ will stand for generic positive constants, which may depend on $d$, $\delta$, and $\alpha$ (where applicable), and nothing else. We will use $C$ for constants whose values can be increased, and $c$ for constants whose values can be decreased, without altering the conclusions. For instance, ``assume $R > C$'' means ``assume $R$ is bigger than a constant depending only on $d$, $\delta$, and $\alpha$'', an assumption that will be implicit in some of our inequalities. 

We know the following from \cite[Theorem 19]{CPPR07}: If $d\ge 5$, then for any $0 < q< p< \infty$ and any $t\ge p^2$,
\begin{eqnarray}
\p\Big(\max_{x\in B(0,1\wedge \frac{q}{2})} \bigl|U\bigl(x \mid B(x,p)\backslash B(x,q)\bigr)\bigr|\ge t \Big) &\le&
C e^{-c q^{d-2} t \log(c t/q^2)}. \label{eq:unif}
\end{eqnarray}
Moreover, by the same theorem, the above bound holds for any $t$ if $p = \infty$. If $d=3$ or $4$, then by \cite[Theorem 36]{CPPR07}, we have that the large deviation bound \eqref{eq:unif} holds for $t\ge C p^2$ provided that $t$ also satisfies 
\begin{equation}
\begin{array}{ll} 
t\ge Cq^2 \log\left(\frac{C t}{q^2}\right)
  \log\left(\frac{p}{q}\right)\qquad\  & \text{in dimension }d=4; \end{array} \label{eq:lrgassumption3}
\end{equation}
\begin{equation}
\begin{array}{ll}
t\ge C q^2 \left(\log\left(\frac{Ct}{q^2}\right)\right)^2
  \log\left(\frac{p}{q}\right)\text{ and } \\
  t\ge C p  q \log\left(\frac{Ct}{q^2}\right)
   & \text{in dimension }d=3.\end{array} \label{eq:lrgassumption4}
\end{equation}
We have the moderate deviation estimate for $d\ge 5$ from Theorem~\ref{moderate_deviation_in_moving_ball_thm}:
\begin{equation}\label{eq:unif2}
\begin{split}
&\p\Big(\max_{x\in B(0,1)} \bigl|U\bigl(x \mid B(x,p)\backslash B(x,q)\bigr)\bigr|\ge t \Big) \\
&\le
C(1+(a/(t\wedge q^2))^d)(1+(t\wedge q^2)^{-d}) e^{-c q^{d-4} (t\wedge q^2)^2}, 
\end{split}
\end{equation}
where $a = p$ if $p <\infty$ and $a=q$ if $p=\infty$. 
Now, let 
\[
r =  R^{-\frac{\delta}{d-2} + \alpha}.
\]
Let $A$ be the ``square annulus'' $Q(0,2R) \backslash Q(0,R)$. Let $M$ be a positive real number, to be chosen later. Divide $A$ into $N := R/Mr$ concentric square annuli of width $Mr$, and call them $A_1,\ldots,A_N$. (Here we assume that $M$ is chosen such that $N$ is an integer. Eventually, the only other requirements on $M$ will be that it is ``large enough'', but smaller than $R^{\frac{\alpha}{2}}$, so this assumption causes no conflict.) Formally,
\[
A_i = Q(0, R + iMr)\backslash Q(0, R+ (i-1)Mr).
\]
If $d \ge 5$, let 
\[
E_0 = \{\text{$\exists x\in A$ such that $|U(x \mid \R^d \backslash B(x,R^{\frac{1}{d-4}}))| \ge R^\delta/2$.}\}
\]
Let $S$ be a collection of points such that $\cup_{x\in S} B(x,1)\supseteq A$. If $\delta > \frac{2}{d-4}$, then by~\eqref{eq:unif} we have
\begin{align*}
\p(E_0) &\le \sum_{x\in S} \p\Big(\max_{y\in B(x,1)} \bigl|U(y\mid \R^d \backslash B(y,R^{\frac{1}{d-4}}))\bigr| \ge R^{\delta}/2\Big) \\
&\le |S| C e^{-c R^{\frac{d-2}{d-4} + \delta}\log R},
\end{align*}
provided $R > C$. 
On the other hand, if $\delta \le \frac{2}{d-4}$, then by \eqref{eq:unif2} we have that for $R > C$, 
\begin{align*}
\p(E_0) &\le \sum_{x\in S} \p\Big(\max_{y\in B(x,1)} \bigl|U(y\mid \R^d \backslash B(y,R^{\frac{1}{d-4}}))\bigr| \ge R^{\delta}/2\Big) \\
&\le |S| CR^Ce^{-cR^{1 + 2\delta} }.
\end{align*}
Now, $S$ can be chosen such that $|S|\le CR^d$. Consequently, we see that in all situations,  if $R>C$ then we have 
\begin{equation}\label{e1bd}
\p(E_0) \le C e^{-c R^{1+\delta}}. 
\end{equation}
Next, if $d\ge 5$, let $\ell_1 = R^{\frac{1}{d-4}}$. If $d =3$ or $4$, let $\ell_1 = \rho$.  Let $\ell_2,\ldots,\ell_K$ be a sequence of numbers such that $R^{\alpha/2} \le \ell_j/\ell_{j+1} \le R^\alpha$ for $j=1,\ldots, K-1$, with $\ell_K = r = R^{-\frac{\delta}{d-2} + \alpha}$. Clearly, it is possible to find such $\ell_j$ with $K$ being an integer bounded by a constant that depends only on $d$, $\delta$, and $\alpha$. That is, in our notation, $K\le C$.  
For each $i = 1,\ldots, N$ and $j = 1,\ldots, K-1$, let
\[
E_i^j := \Big\{\text{$\exists x\in A_i$ such that $|U(x \mid B(x, \ell_j)\backslash B(x,\ell_{j+1}))| \ge \frac{R^\delta}{4(K-1)}$}\Big\}.
\]
Now, choosing $M>2$ and defining the slightly smaller annulus $A_i' = Q(0, R + iMr-r)\backslash Q(0, R+ (i-1)Mr+r)$ for each $i=1,\ldots, N$, we consider the collection $\mathcal{C}_i$ of boxes of the grid $r\Z^d$ which are fully contained in $A_i$. A contiguous sequence of boxes from $\mathcal{C}_i$ is said to cross $A_i'$ if the union of the sequence contains a continuous path crossing $A_i'$. Define 
\begin{align*}
E_i^K &:= \bigl\{\text{$\exists$ a contiguous sequence of boxes from $\mathcal{C}_i$ crossing $A_i'$}\\
&\qquad \qquad \text{such that $\max_{x\in B}|U(x\mid B(x,r))|$ for each box $B$ exceeds $R^\delta/4$}\bigr\}.
\end{align*}
Now, since $\alpha < \frac{\delta}{d-2}$ and $R > C$,  $\max_{x\in B}|U(x\mid B(x,r))| > R^{\delta}/4$ can happen only if $B + B(0,r)$ contains a star, the probability of which is $\le C r^d$. Note that in any self-avoiding chain of $L$ boxes, there are at least $c L$ boxes $B$ such that $B+B(0,r)$ are mutually disjoint. However, $L$ has to be at least $M-2$ in any crossing. Also note that there are at most $(R/r)^{d-1} C^L$ chains of length $L$ that cross $A_i'$. Combining all these observations, we see that if $M > C$ and $R > C$, we have 
\begin{equation}\label{ekbd}
\begin{split}
\p(E_i^K) &\le (R/r)^d \sum_{L\ge M-2} C^L r^{cdL}\\
&\le (R/r)^d (Cr)^{cdM}\le e^{-cM\log R}.
\end{split}
\end{equation}
(Note that we are using, somewhat subtly, the fact that $r < R^{-c}$ since $\alpha < \frac{\delta}{d-2}$.) 
Now, with the above definitions, we clearly have that for $d\ge 5$,
\begin{equation}\label{maineq}
E_{R,\delta} \subseteq E_0 \cup \biggl(\bigcap_{i=1}^N \bigcup_{j=1}^KE_i^j\biggr).
\end{equation}
For $d=3$ and $4$, we have
\begin{equation}\label{maineq2}
E'_{R,\delta, \rho} \subseteq \bigcap_{i=1}^N \bigcup_{j=1}^KE_i^j.
\end{equation}
We already have the bound \eqref{e1bd} for $\p(E_0)$. Let us now bound the probability of the other term. First, note that
\begin{equation}\label{term2}
\begin{split}
\p\biggl(\bigcap_{i=1}^N \bigcup_{j=1}^KE_i^j\biggr) &= \p\biggl(\bigcup_{J:\{1,\ldots,N\}\rightarrow\{1,\ldots,K\}} \bigcap_{i=1}^N E_i^{J(i)}\biggr)\\
&\le \sum_{J:\{1,\ldots,N\}\rightarrow\{1,\ldots,K\}} \p\biggl(\bigcap_{i=1}^N E_i^{J(i)}\biggr)
\end{split}
\end{equation}
(the union and sum are over all functions $J:\{1,\ldots,N\}\to\{1,\ldots,K\}$).
Let us now get some bounds for $\p(E_i^j)$. We already have a bound \eqref{ekbd} for $\p(E_i^K)$, so let us consider $j\le K-1$. Suppose $C\ell_{j} \le R^{\delta/2}$. Then for $d\ge 5$, by the large deviation bound \eqref{eq:unif} (and the same technique as in bounding $\p(E_0)$), we have
\begin{align*}
\p(E_i^j) &\le C(R/r)^d e^{-c\ell_{j+1}^{d-2}R^{\delta}\log R}\\
&\le C(R/r)^d e^{-c\ell_{j+1}R^{\bigl(-\frac{\delta}{d-2} + \alpha\bigr) (d-3) + \delta}\log R} \\
&\le C (R/r)^d e^{-c \ell_j R^{\frac{\delta}{d-2} + \alpha(d-4)}\log R}. 
\end{align*}
Again, if $C\ell_j > R^{\delta/2}$ and $d\ge 5$, we have  $R^{\delta}\wedge \ell_{j+1}^2 \ge R^\delta \wedge (\ell_j^2 R^{-2\alpha})\ge cR^{\delta-2\alpha}$, and therefore  we can use \eqref{eq:unif2} to get
\begin{align*}
\p(E_i^j) &\le CR^Ce^{-c\ell_{j+1}^{d-4}R^{2\delta - 4\alpha}}\\
&\le C R^Ce^{-c\ell_{j}^{d-4}R^{2\delta-d\alpha}}\\
&\le CR^C e^{-c\ell_{j}R^{\frac{\delta(d-1)}{2}-d\alpha}}.
\end{align*}
Since $\alpha < \delta/4$, we have
\[
\frac{\delta(d-1)}{2} - d\alpha > \frac{\delta}{d-2} + \alpha(d-4),
\]
and therefore we can combine the last two relations to conclude that when $d\ge 5$, for any $1\le j\le K-1$ and any $i$ (and $R>C$),
\[
\p(E_i^j) \le C R^Ce^{-c\ell_j R^{\frac{\delta}{d-2} + \alpha(d-4)}\log R}. 
\]
Since $\ell_j \ge R^{\frac{\alpha}{2}}R^{-\frac{\delta}{d-2} + \alpha}$, if $R>C$ this reduces to
\begin{equation}\label{ejbd}
\p(E_i^j) \le C e^{-c\ell_j R^{\frac{\delta}{d-2} + \alpha(d-4)}\log R}. 
\end{equation}
Next, let us consider $d =3,4$. Since $cR^\delta\ge \rho^2 \ge \ell_j^2$ (where in this place only, we may take $c$ depending only on $d$), we can still apply the large deviation bound \eqref{eq:unif} for bounding $\p(E_i^j)$, provided that \eqref{eq:lrgassumption4} holds when $d=3$ and  \eqref{eq:lrgassumption3} holds when $d=4$, with $t= \frac{R^\delta}{4(K-1)}$, $p = \ell_j$ and $q=\ell_{j+1}$. It is easy to see that this happens when $R>C$, since $\ell_{j+1}\le \ell_j R^{-\alpha/2}$ and $\alpha > 0$. Thus, for $R> C$, we have that for $d=3,4$,
\begin{align*}
\p(E_i^j) &\le C(R/r)^d e^{-c\ell_{j+1}^{d-2}R^{\delta}\log R}\\
&\le C(R/r)^d e^{-c\ell_{j+1}R^{\bigl(-\frac{\delta}{d-2} + \alpha\bigr) (d-3) + \delta}\log R} \\
&\le C (R/r)^d e^{-c \ell_j R^{\frac{\delta}{d-2} + \alpha(d-4)}\log R}. 
\end{align*}
Thus, \eqref{ejbd} holds for $d=3,4$ as well.
 
Now fix a map $J:\{1,\ldots,N\} \rightarrow \{1,\ldots, K\}$. We adopt the following procedure for choosing $S\subseteq \{1,\ldots,N\}$ such that the events $(E_i^{J(i)})_{i\in S}$ are mutually independent. First, order the indices $1,\ldots,N$ as $w_1,\ldots,w_N$ such that $\ell_{J(w_1)}\ge \ell_{J(w_2)}\ge \ldots \ge \ell_{J(w_N)}$. Begin constructing $S$ by putting $w_1$ in $S$. Suppose we have inspected $w_1,\ldots, w_{i-1}$. Put $w_i$ in $S$ according to the following rule. If the annulus $A_{w_i} + B(0,\ell_{J(w_i)})$  intersects the union of annuli $A_{w_j} + B(0,\ell_{J(w_j)})$ for $w_j$ that have already been included in $S$, then leave $i$ out, otherwise add it to~$S$. Then by construction, the annuli $\{A_{i} + B(0,\ell_{J(i)})\}_{i\in S}$ are disjoint. Since the event $E_i^j$ depends only on the stars in the annulus $A_i + B(0,\ell_j)$, the events $(E_{i}^{J(i)})_{i\in S}$ are independent. In particular,
\begin{equation}\label{reduc}
\p\biggl(\bigcap_{i=1}^N E_i^{J(i)}\biggr) \le \p\biggl(\bigcap_{i\in S} E_i^{J(i)}\biggr) = \prod_{i\in S} \p(E_i^{J(i)}). 
\end{equation}
Now, the annulus $A_i + B(0,\ell_{J(i)})$ has ``width'' $2\ell_{J(i)} + Mr$ (that is, it is contained in $A_i+Q(0,\ell_{J(i)})$). 
Our construction of $S$ guarantees that for each $i\not \in S$, $A_i+B(0,\ell_{J(i)})$ intersects $A_j+ B(0,\ell_{J(j)})$ for some $j\in S$ such that $\ell_{J(j)} \ge \ell_{J(i)}$. Since the annuli are concentric, it follows that 
\[
A_i + B(0, \ell_{J(i)}) \subseteq A_j + B(0, 3\ell_{J(j)} + Mr). 
\]
Thus, the union of the annuli $\{A_i + B(0,  3\ell_{J(i)} + Mr)\}_{i\in S}$ covers the ``square annulus'' $Q(0,2R)\backslash Q(0,R)$. Now observe that $\ell_{K} = r$ and if $R^{\alpha/2} > M$, we have $\ell_{j} \ge Mr$ for all $j < K$. Thus, defining $S'= \{i\in S: J(i) < K\}$ and $S'' = S\backslash S'$, we get that when $R^{\alpha/2} > M > C$, 
\[
R\le \sum_{i\in S} (3\ell_{J(i)}  + Mr) \le 4\sum_{i\in S'} \ell_{J(i)} + 4|S''|Mr. 
\]
Thus, at least one of the two terms in the rightmost sum has to be $\ge R/2$. First, suppose $4|S''|Mr \ge R/2$. 
Then, from the bound \eqref{ekbd}, we have
\begin{equation}\label{sbd1}
\prod_{i\in S''} \p(E_i^{J(i)}) \le (e^{-cM\log R})^{\frac{R}{8Mr}}\le e^{-cR^{1+\frac{\delta}{d-2} - \alpha}\log R}.
\end{equation}
On the other hand, if 
\[
4\sum_{i\in S'} \ell_{J(i)}\ge  \frac{R}{2},
\]
then by \eqref{ejbd} we have that for $R > C$, 
\begin{equation}\label{sbd2}
\begin{split}
\prod_{i\in S'} \p(E_{i}^{J(i)}) &\le C^N e^{-cR^{\frac{\delta}{d-2}  + \alpha(d-4)}\sum_{i\in S'} \ell_{J(i)}\log R}\\
&\le C^N e^{-cR^{1 + \frac{\delta}{d-2}+\alpha(d-4)}\log R}. 
\end{split}
\end{equation}
From \eqref{term2}, \eqref{reduc}, \eqref{sbd1} and \eqref{sbd2}, we see that if $d\ge 3$ and $R^{\alpha/2} > M > C$ then
\begin{align*}
\p\biggl(\bigcap_{i=1}^N \bigcup_{j=1}^K E_i^j\biggr)&\le \sum_{J:\{1,\ldots,N\}\rightarrow \{1,\ldots, K\}} \p\biggl(\bigcap_{i=1}^N E_i^{J(i)}\biggr) \\
&\le C^N e^{- cR^{1+\frac{\delta}{d-2} - \alpha}\log R}\\
&\le e^{(CM^{-1} - c\log R) R^{1+\frac{\delta}{d-2} - \alpha}}.
\end{align*}
The proof is finished by combining the above bound with \eqref{e1bd}, \eqref{maineq} and \eqref{maineq2}. 
\qed

\end{section}

\begin{section}{Expectation of the force in an empty box}\label{empty_box_force_section}
Recall from Section~\ref{notation_section} the notations $G(x)_i$ and $x_i$. We prove
\begin{proposition}\label{exp_of_force_empty_box_prop}
Let $d\ge 2$, $L,W>0$ and $V:=\Box{L}{W}$. Let
\begin{equation*}
\Omega:=\{\text{$V$ contains no stars}\}.
\end{equation*}
Finally let $G(x):=\E (F(x)|\Omega)$. Then there exist $C,c>0$ independent of $L$ and $W$ such that
\begin{enumerate}
\item[(i)] On the event $\Omega$, $F(x|V)$ is non-random and equals $G$.
\item[(ii)] $|G(x)_1|\le CL^{-(d-2)}W^{d-1}$ when $|x_1|\le \frac{L}{2}$.
\item[(iii)] For $x\in V$ and $2\le i\le d$,
\begin{equation*}
\begin{split}
&G(x)_i\ge cx_i\left(1-\left(\frac{CW}{W+L}\right)^{d-1}\right)\text{ when }x_i\ge0,\\
&G(x)_i\le cx_i\left(1-\left(\frac{CW}{W+L}\right)^{d-1}\right)\text{ when }x_i\le0.
\end{split}
\end{equation*}
\end{enumerate}
\end{proposition}
\begin{proof}
Recalling that
\begin{equation*}
F(x) = F(x|V) + F(x|V^c),
\end{equation*}
that $F(x|V^c)$ is independent of $\Omega$, and that the force is always normalized to have mean 0, we obtain $G(x) = \E(F(x|V)|\Omega)$. But by its definition, on the event $\Omega$
\begin{equation}\label{force_in_empty_box_formula}
F(x|V) = -\int_V \frac{z-x}{|z-x|^d}dz
\end{equation}
which is non-random. This proves part (i).

Now fix $x$ with $|x_1|\le\frac{L}{2}$. Note that when evaluating the first coordinate of formula \eqref{force_in_empty_box_formula} we may ``cancel out'' corresponding parts of the box to the left of $x$ and to its right. More precisely, on the event $\Omega$, if we assume without loss of generality that $x_1\ge 0$ then
\begin{equation*}
F(x|V)_1 = -\int_{V_x^1} \frac{z_1-x_1}{|z-x|^d}dz
\end{equation*}
where $V_x^1:=\{z\ |\ -L\le z_1< 2x_1-L,\ |z_i|\le W\text{ for all $2\le i\le d$}\}$. Hence, since $|x_1|\le \frac{L}{2}$,
\begin{equation*}
|F(x|V)_1|\le \left(\frac{2}{L}\right)^{(d-1)}\vol(V_x^1)\le \left(\frac{2}{L}\right)^{(d-1)}LW^{d-1}=CL^{-(d-2)}W^{d-1},
\end{equation*}
proving part (ii).

We now prove the first part of (iii), the second part follows by symmetry. Fix $2\le i\le d$ and $x$ with $x_i\ge 0$. Similarly to part (ii),
\begin{equation*}
F(x|V)_i = \int_{V_x^i} \frac{x_i-z_i}{|z-x|^d}dz
\end{equation*}
where
\begin{equation*}
V_x^i:=\{z\ |\ -L\le z_1\le L,\ -W\le z_i<2x_i-W,\ |z_j|\le W\text{ for $2\le j\le d$, $j\neq i$}\}.
\end{equation*}
Let also
\begin{equation*}
\tilde{V}_x^i:=\{z\ |\ -L\le z_1\le L,\ -W\le z_i<x_i-W,\ |z_j|\le W\text{ for $2\le j\le d$, $j\neq i$}\}
\end{equation*}
and note that $\tilde{V}_x^i\subseteq V_x^i$. Consider the vertical slice $\{z\ |\ z_1=a\}\cap \tilde{V}_x^i$. On this slice we have
\begin{equation*}
\frac{x_i-z_i}{|z-x|^d}\ge \frac{W}{\left((x_1-a)^2+\sum_{j=2}^d (W+|x_j|)^2\right)^{d/2}} \ge \frac{W}{\left((x_1-a)^2+4(d-1)W^2\right)^{d/2}}.
\end{equation*}
Hence by integrating over $\tilde{V}_x^i$ and estimating the volume of such a slice, we obtain
\begin{equation*}
\begin{split}
F(x|V)_i &\ge x_iW^{d-2}\int_{-L}^L \frac{W}{\left((x_1-a)^2+4(d-1)W^2\right)^{d/2}}da \ge\\
&\ge x_iW^{d-1}\int_{-L}^L \frac{1}{\left((x_1-a)^2+4(d-1)W^2\right)^{d/2}}da \ge\\
&\ge x_iW^{d-1}\int_0^L \frac{1}{\left(4(d-1)W^2+a^2\right)^{d/2}}da\ge\\
&\ge x_iW^{d-1}\int_0^L \frac{1}{(2\sqrt{d-1}W+a)^d}da =\\
&= cx_i\left(1-\left(\frac{CW}{W+L}\right)^{d-1}\right).\qedhere
\end{split}
\end{equation*}

\end{proof}
\end{section}

\begin{section}{Chebyshev-type cubatures}\label{Chebyshev-type_cubature_section}
In the construction of our lower bounds, we will need quantitative theorems estimating how well can the Poisson process approximate a given continuous distribution. The sense of the approximation we will need is that the empirical measure formed by the points of the Poisson process (in some region) has the same (or almost the same) first moments as the continuous distribution (in that region). Such an approximation is called a Chebyshev-type cubature, see \cite{P09} for more information. Specifically, we will need the following:
\begin{enumerate}
\item In Section \ref{small_ball_estimate_section} we will need the fact that when putting many independent uniform points in a cube, the set of configurations forming a Chebyshev-type cubature (with respect to uniform measure on the cube) has positive density, and that in particular, we can lower bound the probability to obtain an approximate Chebyshev-type cubature in this way. 
\item In Section~\ref{second_lower_bound_section} we will need a ``local'' Chebyshev-type cubature formula for a certain measure on the surface of a cylinder. The ``local'' part refers to the fact that we will actually need to partition the cylinder into patches of small diameter and equal volume, and on each patch construct a Chebyshev-type cubature formula (for the measure restricted to that patch) such that the number of points in each of these formulas is uniformly bounded. 
\end{enumerate}
The theorems we need are proven in \cite{P09} and we cite them below. Recall that $P_k^d$ and $\polydim(k,d)$ were defined in Section~\ref{notation_section}.
\begin{theorem} \label{positive_density_thm}
Fix $k\ge 1$ and let $(X_i)_{i=1}^\infty$ be an IID sequence of RV's uniform on $[-1,1]^d$. Let $M_i:=P_k^d(X_i)$ and $\bar{S}_n:=\frac{1}{\sqrt{n}}\sum_{i=1}^n (M_i-\E M_1)$. Then there exists $N_0=N_0(k,d)>0$, $a=a(k,d)>0$ and $t=t(k,d)>0$ such that for all $n>N_0$, $\bar{S}_n$ is absolutely continuous with respect to Lebesgue measure in $\R^{\polydim(k,d)}$ and its density $f_n(x)$ satisfies $f_n(x)\ge a$ for $|x|\le t$.
\end{theorem}
To state our theorem for the cylinder we make a few more definitions.
Given $L,W>0$ and a dimension $d\ge 1$, let 
\begin{equation*}
 P_{L,W}:=\{x\in\R^d\ |\ |x_1|\le L, x_2^2+\cdots +x_d^2=W^2\}.
\end{equation*}
so that $P_{L,W}$ is the curved part of the boundary of a length $L$ cylinder of radius $W$.
Let $\nu_{L,W}$ be the measure supported on $P_{L,W}$ and absolutely continuous with respect to $\sigma_{d-1}$ with density $V(x_1,\ldots, x_d)=v(x_1)=1+\frac{x_1+L}{2L}$. I.e., the density increases linearly from $1$ to $2$ as $x_1$ increases from $-L$ to $L$. Define for $d,k\ge 1$ and $\delta>0$,
\begin{equation}\label{m_0_definition}
m_0(d,k,\delta):=\text{Smallest integer $m\ge 1$ satisfying }\left(\frac{ke}{m+1}\right)^{m+1}\le \frac{\delta}{2d2^k}.
\end{equation}
\begin{theorem}\label{special_cubature_on_cyl_thm}
For each $d\ge 3$ there exists $C>0$ such that for each $k\ge 1$, $L>C$, $W>0$, $0<\tau<W$ and $0<\delta<W^k$ we have measurable subsets $D_1, \ldots, D_{K}\subseteq P_{2L,W}$ satisfying the following properties:
\begin{enumerate}
\item[(I)] $\nu_{2L,W}(D_i\cap D_j)=0$ for each $i\neq j$ and $\nu_{2L,W}\left(P_{L,W}\setminus\left(\cup_{i=1}^{K} D_i\right)\right)=0$.
\item[(II)] $\diam(D_i)\le C\tau$, $\nu_{2L,W}(D_i)= \tau^{d-1}$ for all $i$ and $K\le CLW^{d-2}\tau^{-(d-1)}$.
\item[(III)] For $n=n_1^{d-1}$ where $n_1$ can be any integer satisfying $n_1\ge C^{m_0(d-2,k,\delta/W^k)}$ and for each $1\le i\le K$, there exist $(w_{D_{i},j})_{j=1}^n\subseteq D_i$ such that
\begin{equation*}
\left|\frac{1}{n}\sum_{j=1}^n h(w_{D_i,j}) - \frac{1}{\nu_{2L,W}(D_i)}\int_{D_i} h(w)d\nu_{2L,W}(w)\right|\le \delta
\end{equation*}
for all $h:\R^{d}\to\R$ of the form $h(w)=(w-y)^\alpha$ for $y\in P_{2L,W}$ and a multi-index $\alpha$ with $|\alpha|\le k$.
\end{enumerate}
\end{theorem}

\end{section}

\begin{section}{The Taylor expansion of the force}\label{Taylor_expansion_section}
In this section we consider the function $g$ introduced in \eqref{def_of_g} and develop it in a Taylor series around a fixed $y\neq 0$. For each multi-index $\alpha\in (\N\cup\{0\})^d$ let $a_\alpha\in\R^d$ denote the Taylor coefficient of $g$ around $y$.

\begin{theorem}\label{Taylor_expansion_thm}
There exists $C_{20}>0$ such that
\begin{enumerate}
\item For each multi-index $\alpha$ we have
\begin{equation*}
|a_\alpha|\le \frac{C_{20}}{|y|^{d-1}}\left(\frac{2d}{|y|}\right)^{|\alpha|}.
\end{equation*}
\item For any integer $k\ge 1$ and any $z$ with $|z-y|\le \frac{1}{C_{20}}|y|$ we have
\begin{equation*}
\bigg|g(z) - \sum_{|\alpha|\le k} a_\alpha (z-y)^\alpha\bigg| \le \frac{C_{20}k^d}{|y|^{d-1}} \left(\frac{2d|z-y|}{|y|}\right)^{k+1}.
\end{equation*}
\end{enumerate}
\begin{proof}
Recall that $g(z)=\frac{z}{(\sum_{j=1}^d z_j^2)^{d/2}}$. For the rest of the proof, we will consider $z$ and $y$ as vectors in $\C^d$. Note that $g(z)$ is analytic around $y$ with domain of analyticity containing the poly-disc $\Omega:=\{z\in\C^d\ |\ |z_j-y_j|\le \frac{|y|}{2d}\}$ (in the sense that each coordinate of it is such a function). 
We will use the Cauchy estimates to estimate $a_\alpha$:
\begin{equation*}
\begin{split}
|a_\alpha| &= \left|\frac{1}{(2\pi i)^d}\oint_{|z_1-y_1|=\frac{|y|}{2d}}\cdots\oint_{|z_d-y_d|=\frac{|y|}{2d}} \frac{g(z)\prod_{j=1}^d dz_j}{\prod_{j=1}^d (z_j-y_j)^{\alpha_i+1}}\right| \le\\
&\le \frac{1}{\left(\frac{|y|}{2d}\right)^{|\alpha|}}\max_{z\in\Omega} |g(z)|\le \left(\frac{2d}{|y|}\right)^{|\alpha|} \frac{C_1}{|y|^{d-1}}.
\end{split}
\end{equation*}
It follows that for any $z$ satisfying $|z-y|\le \frac{|y|}{d3^d}$ we have
\begin{equation*}
\begin{split}
\bigg|g(z) - \sum_{|\alpha|\le k} a_\alpha (z-y)^\alpha\bigg| &\le C_1 \sum_{|\alpha|> k} \frac{(2d)^{|\alpha|}}{|y|^{|\alpha|+d-1}} \prod_{j=1}^d |z_j-y_j|^{\alpha_j} \le\\
&\le C_1 \sum_{|\alpha|> k} \frac{(2d)^{|\alpha|}}{|y|^{|\alpha|+d-1}} |z-y|^{|\alpha|} \le\\
&\le \frac{C_1}{|y|^{d-1}}\sum_{l=k+1}^\infty (l+1)^d \left(\frac{2d|z-y|}{|y|}\right)^l\le\\
&\le \frac{C_2 k^d}{|y|^{d-1}}\left(\frac{2d|z-y|}{|y|}\right)^{k+1}.
\end{split}
\end{equation*}
Where the last inequality follows since $d|z-y|\le \frac{|y|}{3^d}$, so the sum is dominated by a geometric series. The theorem follows with $C_{20}=\max(C_1, C_2, d3^d)$.
\end{proof}
\end{theorem}
The following ``deterministic'' proposition which is a corollary of the previous theorem is what we shall be using in the following sections.
\begin{proposition}\label{Taylor_expansion_use_prop}
Let $U\subseteq\R^d$ be bounded measurable and let $Y_1,\ldots, Y_n$ be $d$-dimensional random vectors with $\P(Y_i\in U)=1$ for all $i$. Fix $y\in U$ and consider the event
\begin{equation*}
E_{y,r,t}:=\left\{\max_{\{x\ |\ |x-y|\ge r\}}\left|\sum_{j=1}^n g(Y_j-x) - \E \sum_{j=1}^n g(Y_j-x)\right|\le t\right\}.
\end{equation*}
Let $\rho:=\sup_{z\in U} |z-y|$, fix an integer $k>0$ and let $M_j:=P_k^d(Y_j-y)$. There exists $c_{30}>0$ such that if we let $C_{20}$ be the constant from Theorem~\ref{Taylor_expansion_thm} and if we assume that
\begin{align}
r&>C_{20}\rho\label{r_condition},\\
t&>\frac{3C_{20}nk^d}{r^{d-1}} \left(\frac{2d\rho}{r}\right)^{k+1}\label{t_condition},
\end{align}
then $E_{y,r,t}\supseteq\Omega_{y,r,t}$ where
\begin{equation*}
\Omega_{y,r,t}:=\left\{\left|\sum_{j=1}^n M_j - \E \sum_{j=1}^n M_j\right|\le \frac{c_{30}tr^{d-1}}{\polydim(k,d)^{1/2}}\left(\frac{r}{2d+r}\right)^k\right\}.
\end{equation*}
\end{proposition}
\begin{proof}
Assume \eqref{r_condition}, \eqref{t_condition} and that $\Omega_{y,r,t}$ occurred and fix $x\in\R^d$ with $|x-y|\ge r$. By translating the set $U$ if necessary we assume without loss of generality that $x=0$. Note that this implies that $|y|\ge r$. We develop the function $g(z)$ in a Taylor series around $y$. Taking $z\in U$ and noting that $|y|\ge r>C_{20} \rho\ge C_{20}|z-y|$ we obtain by Theorem~\ref{Taylor_expansion_thm} that
\begin{equation*}
\bigg|g(z) - \sum_{|\alpha|\le k} a_\alpha (z-y)^\alpha\bigg| \le \frac{C_{20}k^d}{r^{d-1}} \left(\frac{2d\rho}{r}\right)^{k+1}< \frac{t}{3n}.
\end{equation*}
Hence
\begin{equation}\label{Taylor_estimate_for_Y_j}
\bigg|\sum_{j=1}^n g(Y_j) - \sum_{j=1}^n\sum_{|\alpha|\le k} a_\alpha (Y_j-y)^\alpha\bigg| < \frac{t}{3}.
\end{equation}
For $|\alpha|\le k$ we have by Theorem~\ref{Taylor_expansion_thm} that
\begin{equation*}
|a_\alpha|\le \frac{C_{20}}{r^{d-1}}\left(\frac{2d}{r}\right)^{|\alpha|}\le \frac{C_{20}}{r^{d-1}}\left(\frac{2d}{r}+1\right)^{k}.
\end{equation*}
Using the Cauchy-Schwartz inequality and $\Omega_{y,r,t}$, this implies that as long as $c_{30}<\frac{1}{3C_{20}}$ we have
\begin{multline}\label{close_Taylor_series}
\bigg|\sum_{j=1}^n\sum_{|\alpha|\le k} a_\alpha (Y_j-y)^\alpha - \E\sum_{j=1}^n\sum_{|\alpha|\le k} a_\alpha (Y_j-y)^\alpha\bigg|\le\\
\le \sum_{0<|\alpha|\le k}|a_\alpha| \bigg|\sum_{j=1}^n (Y_j-y)^\alpha - \E (Y_j-y)^\alpha\bigg|< \frac{t}{3}.
\end{multline}
Finally, using \eqref{Taylor_estimate_for_Y_j} again and the triangle inequality
\begin{equation}\label{expected_Taylor_estimes_for_Y_j}
\bigg|\E\sum_{j=1}^n g(Y_j) - \E\sum_{j=1}^n\sum_{|\alpha|\le k} a_\alpha (Y_j-y)^\alpha\bigg| < \frac{t}{3}.
\end{equation}
Putting \eqref{Taylor_estimate_for_Y_j}, \eqref{close_Taylor_series} and \eqref{expected_Taylor_estimes_for_Y_j} together we get
\begin{equation*}
\bigg|\sum_{j=1}^n g(Y_j) - \E\sum_{j=1}^n g(Y_j)\bigg|< t
\end{equation*}
as required.
\end{proof}
\end{section}

\begin{section}{Small ball estimate for the ``cosmic background noise''}\label{small_ball_estimate_section}
In this section we take a box $\Vs$ which is very long on one side and short on the other sides. We take 
$\Vh$ to be a ``concentric'' cube which is very long on all dimensions. We consider the force in the small box $\Vs$ from the stars in $\Vh\setminus 2 \Vs$, and we prove a lower bound for the probability that this force is extremely close to its expectation. More precisely, throughout this section we fix $0<\eps<\frac{1}{2(d-2)}$. Let $p_1:=\lceil \frac{\log R}{\log 2}\rceil,\ p_2:=\lceil \frac{\eps\log R}{\log 2}\rceil$ and define
\begin{equation*}
\begin{split}
&\Vs := \Box{2^{p_1}}{2^{2p_2}},\\
&\Vh := \Box{2^{p_1+1}}{2^{p_1+1}}.
\end{split}
\end{equation*}
We note that
\begin{equation}\label{estimates_on_p_i}
\begin{split}
R\ \le\ &2^{p_1} \le 2R,\\
R^{\eps} \le\ &2^{p_2} \le 2R^{\eps}.
\end{split}
\end{equation}
The idea is that $2^{p_1}$ is approximately $R$ and $2^{p_2}$ is approximately $R^{\eps}$, but for technical reasons we need these dimensions to be integer powers of $2$. We shall prove:
\begin{theorem}\label{small_ball_estimate_thm}
There exists $C(\eps)$ such that for $R\ge C(\eps)$,
\begin{equation*}
\P(\max_{x\in \Vs} |F(x\ |\ \Vh\setminus 2\Vs)|\le\frac{1}{R^d}) \ge \exp(-C(\eps)R^{1+(d-2)\eps}\log R).
\end{equation*}
\end{theorem}

We remark that the fact that the bound on $|F(x\ |\ \Vh\setminus 2\Vs)|$ given by the theorem is $\frac{1}{R^d}$ is not essential for proving the theorem; putting a higher power of $\frac{1}{R}$ there would only affect the constants in the probabilistic estimate.

{\bf Sketch of proof:} The proof works by dividing $\Vh\setminus 2\Vs$ into $CR^{1+(d-2)\eps}$ cubes with the diameter of each cube a little smaller than its distance from $\Vs$. We then rely on Theorem~\ref{positive_density_thm} (where most of the work is) to say that for each cube, with probability at least $R^{-C}$, the stars in the cube approximate Lebesgue measure in the cube in the sense that if $(z_i)$ are these stars then $\sum \delta_{z_i}$ has its first multi-moments very close to those of the Lebesgue measure. By considering the Taylor expansion of the force (Section~\ref{Taylor_expansion_section}) and using the diameter condition we observe that such an approximation suffices to control the force in $\Vs$.

For the rest of the proof we fix a constant $A=A(\eps)>0$, large enough as needed for the proof of Proposition~\ref{force_from_box_prop} below. For a point $y\in\R^d$ we denote $d(y,\Vs):=\max_{x\in \Vs} |x-y|$. We start with a definition.
\begin{definition}
We call a cube $B=y+[-a/2,a/2]^d \subset\R^d$ {\bf dominated by $\Vs$} if 
the side length $a$ of $B$ is an integer and satisfies $A\le a\le 2^{-p_2} d(y,\Vs)$.
\end{definition}
Note that from the definition, if $B$ is dominated by $\Vs$ then its center $y$ has to satisfy $d(y,\Vs)\ge A2^{p_2}$.

Theorem~\ref{small_ball_estimate_thm} will follow from the following two propositions.
\begin{proposition}\label{box_partition_prop}
There exists $C(\eps)$ such that for $R\ge C(\eps)$, the set $\Vh\setminus 2\Vs$ may be partitioned into $n\le 2^{3d}R^{1+(d-2)\eps}$ cubes $(B_i)_{i=1}^n$ which are dominated by $\Vs$ (the cubes are disjoint except for their boundaries).
\end{proposition}

\begin{proposition}\label{force_from_box_prop}
There exists $C(\eps)>0$ such that for $R\ge C(\eps)$ and any cube $B\subseteq \Vh$, dominated by $\Vs$,
\begin{equation*}
\P\left(\max_{x\in \Vs} |F(x\ |\ B)|\le \frac{1}{R^{d+2}}\right)\ge R^{-C(\eps)}.
\end{equation*}
\end{proposition}
We first show how the theorem follows from these two propositions, then we prove Proposition~\ref{box_partition_prop} and finally Proposition~\ref{force_from_box_prop}.
\begin{proof}[Proof of Theorem~\ref{small_ball_estimate_thm}]
We use Proposition~\ref{box_partition_prop} to obtain the cubes $(B_i)_{i=1}^n$ which partition $\Vh\setminus 2\Vs$. We have (except on the negligible event where the boundary of some $B_i$ contains a star)
\begin{equation*}
F(x\ |\ \Vh\setminus 2\Vs) = \sum_{i=1}^n F(x\ |\ B_i).
\end{equation*}
Define for each $1\le i\le n$ the event
\begin{equation*}
E_i:=\left\{\max_{x\in \Vs} |F(x\ |\ B_i)|\le \frac{1}{R^{d+2}}\right\}.
\end{equation*}
Since the boxes are dominated by $\Vs$, Proposition~\ref{force_from_box_prop} gives that so long as $R\ge C(\eps)$, for each $i$
\begin{equation*}
\P(E_i)\ge R^{-C(\eps)}.
\end{equation*}
Define $E:=\{\max_{x\in \Vs} |F(x\ |\ \Vh\setminus 2\Vs)|\le\frac{1}{R^d})\}$ and note that $E\supseteq\cap_{i=1}^n E_i$ if $R$ is large enough since $n\le 2^{3d}R^{1+(d-2)\eps}$ and $\eps<\frac{1}{2(d-2)}$. Noting further that the events $E_i$ are independent since they depend only on the points of the Poisson process in disjoint boxes, we deduce
\begin{equation*}
\P(E)\ge R^{-C(\eps)n}\ge \exp(-C(\eps)2^{3d}R^{1+(d-2)\eps}\log R)
\end{equation*}
as required.
\end{proof}

\begin{subsection}{Proof of Proposition~\ref{box_partition_prop}}
For this proof, we shall consider $\Z^d$ as a grid of cubes of side length $1$, and $a\Z^d$ as the stretched grid with side length $a$. The cubes $(B_i)_{i=1}^n$ which we shall exhibit will be a subset of the cubes of $a_i\Z^d$ for various values of $a_i$. Indeed, let us fix a sequence of scales
\begin{equation*}
s_i:=2^{2p_2+i}\text{ for }1\le i\le p_1-2p_2+1.
\end{equation*}
and a sequence of side lengths
\begin{equation*}
a_i:=2^{p_2+i-1}\text{ for }1\le i\le p_1-2p_2+1.
\end{equation*}
Consider the boxes $V_i:=\Box{2^{p_1+1}}{s_i}$, note that $V_1=2\Vs$ and $V_{p_1-2p_2+1}=\Vh$. For $1\le i\le p_1-2p_2$ let $\mathcal{C}_i$ be the subset of cubes $a_i\Z^d$ which are fully contained in $V_{i+1}\setminus V_i$. Finally, the set of cubes $(B_j)_{j=1}^n$ is the union of all the $\mathcal{C}_i$. It is straightforward to see that the $(B_j)_{j=1}^n$ partition the set $\Vh\setminus 2\Vs$ (except that their boundaries may overlap). It remains to check that the cubes are dominated by $\Vs$ and to check the estimate on $n$. To check the former we note that the cubes have their sides parallel to the axes by construction and that their side lengths $a_i$ are integers larger than $A$ once $R$ is large enough (as a function of $\eps$). In addition, note that for any point $y\in V_{i+1}\setminus V_i$ we have $d(y,\Vs)\ge s_i-2^{2p_2}\ge \frac{s_i}{2}$. Hence, in particular, the center $y$ of each cube $B_i$ in $\mathcal{C}_i$ satisfies this, from whence it follows that the side length $a_i$ of the cube satisfies $a_i=2^{-p_2-1}s_i\le 2^{-p_2} d(y,\Vs)$, proving that the cube is dominated by $\Vs$. To check the estimate on $n$ we note that by \eqref{estimates_on_p_i}
\begin{equation*}
|\mathcal{C}_i|\le \frac{\vol(V_{i+1})}{a_i^d} = 2^{p_1+1}s_{i+1}^{d-1}a_i^{-d} \le R2^{(d-2)p_2+2d-i+1}.
\end{equation*}
Hence $n=\sum_{i=1}^{p_1-2p_2} |\mathcal{C}_i|\le R2^{(d-2)p_2+2d+1}\le 2^{3d}R^{1+(d-2)\eps}$ as required.
\end{subsection}
\begin{subsection}{Proof of Proposition~\ref{force_from_box_prop}}
Let $B$ be a cube which is dominated by $\Vs$. Let $y$ be the center point of $B$, let $a$ be the side length of $B$ and let $v:=a^d$ be the volume of $B$. Let $N$ be the number of stars in $B$ and define the event (recalling that $v$ is an integer)
\begin{equation*}
\Omega_1:=\{N=v\}.
\end{equation*}
Then by Lemma~\ref{poissonlemma} and the assumption that $B\subseteq \Vh$,
\begin{equation}\label{prob_of_box_with_expected_stars}
\P(\Omega_1)\ge \frac{c}{\sqrt{v}} \ge \frac{c}{R^{d/2}}.
\end{equation}
We condition on the event $\Omega_1$ and let $Y_1,\ldots, Y_v$ be the stars in $B$ in uniform random order (so that $(Y_i)_{i=1}^v$ are distributed as IID uniform vectors in $B$). Recalling the definition of $g$ from \eqref{def_of_g}, we note that for any $x\in\R^d$, on the event $\Omega_1$ we have
\begin{equation*}
F(x\ |\ B) = \sum_{i=1}^v g(Y_i-x) - v\E g(Y_1-x) = \sum_{i=1}^v g(Y_i-x) - \E \sum_{i=1}^v g(Y_i-x) .
\end{equation*}
Hence denoting $r:=d(y,\Vs)$, $t:=\frac{1}{R^{d+2}}$ and
\begin{equation*}
E:=\left\{\max_{\{x\ |\ |x-y|\ge r\}} \left|\sum_{i=1}^v g(Y_i-x) - \E \sum_{i=1}^v g(Y_i-x)\right|\le t\right\},
\end{equation*}
we see that to prove the proposition it will be enough to show that for large enough $R$ (as a function of $\eps$)
\begin{equation}\label{cond_noise_from_box_estimate}
\P(E\ |\ \Omega_1)\ge R^{-C(\eps)},
\end{equation}
since by \eqref{prob_of_box_with_expected_stars} the price of conditioning on $\Omega_1$ is also at most a bounded negative power of $R$.
We shall use Proposition~\ref{Taylor_expansion_use_prop} with $U=B, n=v$ and $k=\lceil\frac{2d+3}{\eps}\rceil$ to prove this estimate.
The proposition implies that under conditions~\eqref{r_condition}, \eqref{t_condition} we have
\begin{equation} \label{E_contains_Omega_2}
E\cap\Omega_1\supseteq\Omega_2\cap\Omega_1.
\end{equation}
where
\begin{equation*}
\Omega_2:=\left\{\left|\sum_{j=1}^v M_j - \E \sum_{j=1}^v M_j\right|\le \frac{c_{30}tr^{d-1}}{\polydim(k,d)^{1/2}}\left(\frac{r}{2d+r}\right)^k\right\}.
\end{equation*}
and $M_j:=P_k^d(Y_j-y)$. Let us first check the conditions. Note that
\begin{equation*}
r=d(y,\Vs)\ge a2^{p_2}\ge R^{\eps}a
\end{equation*}
by the fact that $B$ is dominated by $\Vs$, so condition \eqref{r_condition} certainly holds if $R$ is large enough (as a function of $\eps$). To check condition \eqref{t_condition} we need to verify that
\begin{equation}\label{t_condition_in_use}
t=\frac{1}{R^{d+2}}>\frac{3C_{20}nk^d}{r^{d-1}}\left(\frac{2d\rho}{r}\right)^{k+1},
\end{equation}
where $\rho=\sup_{z\in B}|z-y|=ca$. Noting that $\frac{\rho}{r}\le CR^{-\eps}$, that $r>A2^{p_2}\ge AR^{\eps}$ by the fact that $B$ is dominated by $\Vs$ and that since $B\subseteq \Vh$ we have $n=v\le CR^d$ we see that \eqref{t_condition_in_use} holds if $R$ is large enough (as a function of $\eps$) by our choice of $k$.

We deduce from \eqref{cond_noise_from_box_estimate} and \eqref{E_contains_Omega_2} that the proposition will be proved by showing that for large enough $R$ (as a function of $\eps$)
\begin{equation}\label{large_box_moment_prob_estimate}
\P(\Omega_2\ |\ \Omega_1) \ge R^{-C(\eps)} .
\end{equation}
Define the affine transformation $T$ that transforms the cube $B$ into the cube $[-1,1]^d$. Note that $T(Y_1),\ldots, T(Y_v)$ are uniform on $[-1,1]^d$. Define $\tilde{M_j}:=P_k^d(T(Y_j))$. Noting that for each $\alpha$ we have $(Y_j-y)^\alpha=a^{|\alpha|} T(Y_j)^\alpha$, we deduce that for each $0<|\alpha|\le k$
\begin{equation*}
|(Y_j-y)^\alpha - \E(Y_j-y)^\alpha|=a^{|\alpha|}|T(Y_j)^\alpha - \E T(Y_j)^\alpha|\le a^k|T(Y_j)^\alpha - \E T(Y_j)^\alpha|
\end{equation*}
since $a\ge A$ and we may take $A\ge 1$. It follows that
\begin{equation*}
\bigg|\sum_{j=1}^v M_j-\sum_{j=1}^v \E M_j\bigg|\le a^k\bigg|\sum_{j=1}^v \tilde{M}_j - \sum_{j=1}^v \E \tilde{M}_j\bigg|.
\end{equation*}
Hence \eqref{large_box_moment_prob_estimate} will follow from
\begin{equation}\label{small_box_moment_prob_estimate}
\P(\Omega_3\ |\ \Omega_1) \ge R^{-C(\eps)}
\end{equation}
where
\begin{equation*}
\Omega_3:=\left\{\left|\sum_{j=1}^v \tilde{M}_j - \sum_{j=1}^v \E \tilde{M}_j\right|\le \frac{c_{30}tr^{d-1}}{\polydim(k,d)^{1/2}}\left(\frac{r}{(2d+r)a}\right)^k\right\}.
\end{equation*}
Note that for $R$ large enough as a function of $\eps$,
\begin{equation*}
\frac{c_{30}tr^{d-1}}{\polydim(k,d)^{1/2}}\left(\frac{r}{(2d+r)a}\right)^k\ge R^{-\tilde{C}(\eps)}
\end{equation*}
where $\tilde{C}(\eps)>0$. Hence, estimate \eqref{small_box_moment_prob_estimate} follows from Theorem~\ref{positive_density_thm} so long as $v>N_0(k,d)$. Recalling that $v=a^d\ge A^d$ by the fact that $B$ is dominated by $\Vs$, we see this happens so long as $A$ is chosen large enough as a function of $\eps$.
\end{subsection}
\end{section}

\begin{section}{Proof of the main theorem - lower bound}\label{lower_bound_section}
Consider the cell of the allocation containing the origin, and let $Z_R$ be the volume of this cell remaining after the intersection of it with a ball of radius $R$ around the star was removed. In a formula,
\begin{equation*}
Z_R:=\vol(\psi_{\cal Z}^{-1}(\psi_{\cal Z}(0))\setminus B(\psi_{\cal Z}(0),R)).
\end{equation*}
We aim to give a lower bound for $\P(Z_R>\exp(-R^\gamma))$ for various values of $\gamma$. We will do this by explicitly constructing an event on which $Z_R>\exp(-R^\gamma)$ and estimating the probability of our construction. Recall from Section~\ref{notation_section} the definitions of $F(x)_1$, $F(x)_n$ and $\partial'$. The main proposition we shall need (a version of which was also implicitly used in \cite{NSV07}, see also Figure~\ref{fig-cylinder1}) is
\begin{proposition}\label{cylinder_construction_prop}
 Let $V:=\Cyl{R}{M}$ for some $M>0$ and for $\xi>0$ let
 \begin{equation}\label{force_condition_for_lower_bound}
 E_1^{\xi}:=\{\forall x\in V\quad \xi R^{1-\gamma}\ge F(x)_1\ge \xi^{-1} R^{1-\gamma}\}\cap\{\min_{x\in\partial' V} F(x)_n>0\}.
 \end{equation}
 Then there exist $C(\xi),c(\xi)>0$ so that if $V\subseteq \Box{R^{2d}}{R^{2d}}$ and $R\ge C(\xi)$ then
 \begin{equation*}
 \P\left(Z_{c(\xi)R}>\frac{M^{d-1}}{R^{C(\xi)}}\exp(-C(\xi)R^{\gamma})\right)\ge \frac{\P(E_1^{\xi})-\exp(-R^{2d}\log^{-C(\xi)} R)}{R^{C(\xi)}}
 \end{equation*} 
\end{proposition}
We remark that if $E_1^{\xi}$ occurs then in particular $V$ contains no stars. The proof is based on Liouville's theorem and also uses Fubini's theorem and the upper bound from \cite[Th. 3]{CPPR07}.

\begin{proof}
Consider a ``slab'' of $V$,
\begin{equation*}
W:=\{x\ |\ |x_1|\le1\}\cap V.
\end{equation*}
Noting that $\vol(W)\ge cM^{d-1}$, we let $E_2^\sigma$ be the event that there is a star whose cell intersects $W$ in a set of volume at least $\frac{M^{d-1}}{R^\sigma}$, and in addition, that the cell of this star is fully contained in $\Box{R^\sigma}{R^\sigma}$. To estimate the probability of $E_2^\sigma$, define
\begin{equation*}
\begin{split}
\Omega_3:=\{&\text{There are not more than $4^{d+1}R^{2d^2}$ stars in $\Box{2R^{2d}}{2R^{2d}}$}\},\\
\Omega_4:=\{&\text{There is no gravitational flow curve connecting}\\
&\text{$\partial \Box{R^{2d}}{R^{2d}}$ and $\partial \Box{2R^{2d}}{2R^{2d}}$}\},\\
\Omega_5:=\{&\text{There is no gravitational flow curve connecting}\\
&\text{$\partial \Box{2R^{2d}}{2R^{2d}}$ and $\partial \Box{4R^{2d}}{4R^{2d}}$}\}.
\end{split}
\end{equation*}
We note that since $V\subseteq \Box{R^{2d}}{R^{2d}}$ we have $E_2^\sigma\supseteq\Omega_3\cap\Omega_4\cap\Omega_5$ if $R\ge C$ and $\sigma$ is a large enough constant. By Lemma~\ref{poissonlemma} we have $\P(\Omega_3)\ge1-C\exp(-cR^{2d^2})$ and by \cite[Th. 3]{CPPR07}, one of the main results of \cite{CPPR07}, we know that $\P(\Omega_4\cap\Omega_5)\ge1-C\exp(-cR^{2d}/\log^C(R))$. This implies that for large enough $\sigma$, $\P(E_2^\sigma)\ge1-C\exp(-cR^{2d}/\log^C(R))$. Fixing such $\sigma$ we see that for $R\ge C$ we have
\begin{equation*}
\P(E_1^\xi, E_2^\sigma)\ge \P(E_1^{\xi})-\exp(-R^{2d}\log^{-C} R).
\end{equation*}

Now let $E_3^\tau$ be the event that there is a star $z$ whose cell is completely contained in $\Box{R^{\tau}}{R^{\tau}}$, such that the volume of its cell minus $B(z,\tau^{-1} R)$ is larger than $\frac{M^{d-1}}{R^{\tau}}e^{-\tau R^{\gamma}}$. We will show that for $\tau\ge C(\xi)$,
\begin{equation*}
E_3^\tau\supseteq E_1^\xi\cap E_2^\sigma.
\end{equation*}
Indeed, assume that $E_1^\xi$ and $E_2^\sigma$ have both occurred. Let $A\subseteq W$ be a measurable set of volume at least $M^{d-1}R^{-\sigma}$ which is allocated to a star $z$ whose cell is fully contained in $\Box{R^\sigma}{R^\sigma}$. Let $A^t$ be the backward flow of $A$ for $t$ time units through the gravitational flow, i.e., $A^t$ consists of all points from which if you flow along the gravitational flow curve for $t$ time units, you end up in $A$. Recall that by Lemma~\ref{Liouville_thm_lemma}, the version of Liouville's theorem, $\vol(A^t)=e^{-d\kappa_dt}\vol(A)\ge M^{d-1}e^{-d\kappa_dt}R^{-\sigma}$. Since the force in $V$ satisfies the estimates given by $E_1^\xi$, we deduce that if $t=c_1(\xi)R^\gamma$ for a small enough $c_1(\xi)>0$, then $A_t\subseteq\{-R<x_1<-c_2(\xi)R\}\cap V$ for a suitable $c_2(\xi)>0$. Since either $A$ or $A_t$ must then lie outside a ball of radius $c_3(\xi) R$ around $z$ for some $c_3(\xi)>0$, we deduce that the event $E_3^\tau$ for $\tau\ge C(\xi)$ has occurred. In conclusion, fixing such a $\tau$ and taking $R\ge C$ we have
\begin{equation*}
\P(E_3^\tau)\ge \P(E_1^{\xi})-\exp(-R^{2d}\log^{-C} R).
\end{equation*}

To conclude, we use a Fubini-type argument to say that if with some probability, a cell not too far from the origin satisfies a certain property, then the cell of the origin satisfies the same property with a probability which is not much lower. More precisely, let $E_4^x$ for $x\in\R^d$ be the event that the cell containing $x$ has volume larger than $\exp(-\tau R^{\gamma})R^{-\tau}$ outside a ball of radius $\tau^{-1} R$ around its star. Let $B:=\Box{R^{\tau}}{R^{\tau}}$, then by translation equivariance
\begin{equation*}
\vol(B)\P(E_4^0)=\int_B\P(E_4^x)dx=\E\int_B 1_{E_4^x}dx\ge \E 1_{E_3^{\tau}} \int_B 1_{E_4^x}dx\ge\P(E_3^{\tau}).
\end{equation*}
Hence for $R\ge C$,
\begin{equation*}
\P(E_4^0)\ge \frac{\P(E_1^{\xi})-\exp(-R^{2d}\log^{-C} R)}{R^{C(\xi)}}
\end{equation*}
for some $C(\xi)>0$. This is the required estimate.
\end{proof}

In the rest of the section we shall present two constructions which lower bound the probability of the event $E_1^\xi$ for different regimes of $d$ and $\gamma$. The first construction (``attracting galaxy'') will give the lower bound for Theorem~\ref{main_theorem} for $d=3$, $0\le \gamma\le 1$ and for $d=4$, $\frac{4}{3}\le \gamma\le \frac{3}{2}$. The second construction (``wormhole'') will give the lower bound for $d=4$, $0\le \gamma\le \frac{4}{3}$ and for $d=5$, $0\le \gamma\le 2$. The constructions differ in whether the required estimate on $F(x)_1$ is due to the effects of ``far away'' stars or ``nearby'' stars.

The constructions have the following in common. We fix $0<\eps<\frac{1}{10d}$ and as in Section~\ref{small_ball_estimate_section} let $p_1:=\lceil \frac{\log R}{\log 2}\rceil,\ p_2:=\lceil \frac{\eps\log R}{\log 2}\rceil$ and
\begin{equation*}
\begin{split}
&\Vs := \Box{2^{p_1}}{2^{2p_2}},\\
&\Vh := \Box{2^{p_1+1}}{2^{p_1+1}}.
\end{split}
\end{equation*}

\begin{subsection}{First lower bound construction - Attracting galaxy}
In this section we shall prove
\begin{theorem}\label{far_stars_mass_bound_thm}
For all dimensions $d\ge 3$ and $0<\eps<\frac{1}{10d}$, there exist $C,c,C(\eps)>0$ such that for all $R\ge C(\eps)$ and $0\le\gamma\le \frac{d-1}{2}$
\begin{equation*}
\P(Z_R>\exp(-CR^\gamma)/R^C)\ge c\exp(-CR^{1+2\eps(d-1)}-CR^{d-2\gamma}).
\end{equation*}
\end{theorem}

In addition to the common parts of the constructions, we define two more sets (see also Figure~\ref{fig-cylinder-first})
\begin{equation*}
\begin{split}
&\Vt := \Cyl{R}{M},\\
&U :=\Cyl{R}{\eta R} + 10Re_1,
\end{split}
\end{equation*}
where $e_1$ is a unit vector in the first coordinate direction, $M$ is a very large constant chosen in Corollary~\ref{force_in_cylinder_for_small_d_cor} below, and $0.5<\eta=\eta(R)<1$ is some number chosen so that $\vol(U)$ is an integer (we assume $R$ is large). To ensure that $\Vt\subseteq \Vs$, with some margin, we always assume below that $R$ is large enough so that
\begin{equation}\label{R_and_M_relation}
R^{2\eps}\ge 2M.
\end{equation}
We define two events
\begin{equation*}
\begin{split}
\Omega_1&:=\{\text{$2\Vs$ contains no stars}\},\\
\Omega_2^k&:=\{\text{$U$ contains exactly $\vol(U)+k$ stars}\},
\end{split}
\end{equation*}
where $k\ge0$ is an integer. Note that
\begin{equation}\label{omega_1_prob_estimate}
\P(\Omega_1)=\exp(-\vol(2\Vs))\ge \exp(-CR^{1+2\eps(d-1)}).
\end{equation}
and by Lemma~\ref{poissonlemma}
\begin{equation}\label{omega_2_k_prob_estimate}
\P(\Omega_2^k)\ge \frac{c}{\sqrt{R^d+k}}\exp\left(-C\frac{k^2}{R^d}\right).
\end{equation}

We divide the force into four parts
\begin{equation*}
F(x)=\underbrace{F(x\ |\ 2\Vs)}_{F^1(x)} + \underbrace{F(x\ |\ \Vh\setminus 2\Vs)}_{F^2(x)} + \underbrace{F(x\ |\ \R^d\setminus(\Vh\cup U))}_{F^3(x)} + \underbrace{F(x\ |\ U)}_{F^4(x)}.
\end{equation*}
We state a proposition we shall use when proving Theorem~\ref{far_stars_mass_bound_thm}.
\begin{proposition}\label{probability_of_atypical_force_prop}
For $0<\eps<\frac{1}{10d}$, there exist $C_{10},C_{11},C,c,C(\eps),c(\eps)>0$ such that if $R\ge C(\eps)$, $M\ge 1$ and relation \eqref{R_and_M_relation} holds, then:
\begin{enumerate}
\item On the event $\Omega_1$, we have deterministically that
\begin{equation*}
\begin{split}
\forall &x\in \Vt\quad |F^1(x)_1|<\frac{C}{R^{d-5/2}},\\
\forall &x\in \partial' \Vt\quad F^1(x)_n>cM.
\end{split}
\end{equation*}
\item \begin{equation*}
\P\left(\max_{x\in \Vt} |F^2(x)|\le \frac{1}{R^d}\right)\ge c(\eps)\exp(-C(\eps)R^{1+\eps(d-2)}\log R).
\end{equation*}
\item \begin{equation*}
\P\left(\max_{x\in \Vt} |F^3(x)|\le \frac{C_{10}\log^{1/2}(R)}{R^{(d-2)/2}}\right)\ge \frac{1}{2}.
\end{equation*}
\item For all integer $k$ such that $C_{11}\log^{1/2}(R)R^{d/2}<k<M^{-1}R^d$
\begin{multline*}
\P\Bigg[\left\{\forall x\in \Vt\quad C\frac{k}{R^{d-1}}>F^4(x)_1>c\frac{k}{R^{d-1}}\right\}\bigcap\\
\left\{\max_{x\in \partial' \Vt} |F^4(x)_n|\le C\right\}\ \bigg|\ \Omega_2^k\Bigg]\ge c.
\end{multline*}
\end{enumerate}
\end{proposition}
We first show how this proposition is used and then we present its proof.
\begin{corollary}\label{force_in_cylinder_for_small_d_cor}
We may choose $\xi,C,c>0$, the constant $M$ from the definition of $\Vt$ (independently of $R$ and $\eps$) and $C(\eps)>0$ such that for all $0\le\gamma\le \frac{d-1}{2}$, letting
\begin{equation*}
E_1^{\xi}:=\{\forall x\in \Vt\quad \rho R^{1-\gamma}\ge F(x)_1\ge \rho^{-1} R^{1-\gamma}\}\cap\{\min_{x\in\partial' \Vt} F(x)_n>0\},
\end{equation*}
we have for $R\ge C(\eps)$ that
\begin{equation}\label{large_forces_in_V_t_prob_estimate}
\P(E_1^\xi)\ge c\exp(-CR^{1+2\eps(d-1)}-CR^{d-2\gamma}).
\end{equation}
\end{corollary}
\begin{proof}
Choose $k=\frac{1}{2}M^{-1}\lfloor R^{d-\gamma}\rfloor$ ($M$ will be chosen shortly below). Note that $\Omega_1$ and $\Omega_2^k$ are independent. By \eqref{omega_1_prob_estimate} and \eqref{omega_2_k_prob_estimate} we have
\begin{equation}\label{omega_1_omega_2_k_prob_estimate}
\begin{split}
\P(\Omega_1, \Omega_2^k)&\ge \frac{c}{\sqrt{R^d+k}}\exp(-CR^{1+2\eps(d-1)}-C\frac{k^2}{R^d})\ge\\
&\ge c\exp(-CR^{1+2\eps(d-1)}-CR^{d-2\gamma})
\end{split}
\end{equation}
for $R\ge C$. Note that if we let $\mathcal F_1:=\sigma(F^1, \Omega_1), \mathcal F_2:=\sigma(F^2), \mathcal F_3:=\sigma(F^3)$ and $\mathcal F_4:=\sigma(F^4, \Omega_2^k)$ where $\sigma(\cdot)$ denotes the $\sigma$-field generated by a family of events and/or random variables, then ${\mathcal F}_1$, ${\mathcal F}_2$, ${\mathcal F}_3$, and ${\mathcal F}_4$ are independent $\sigma$-fields. 
Hence, by Proposition \ref{probability_of_atypical_force_prop} (combining all 4 parts) we see that we may choose $M$ large enough and then $\xi$ large enough as a function of $M$ such that if $R\ge C(\eps)$ then
\begin{equation*}
\P(E_1^\xi\ |\ \Omega_1, \Omega_2^k)\ge c(\eps)\exp(-C(\eps)R^{1+\eps(d-2)}\log R).
\end{equation*}
Fixing $\xi$ and $M$ for which this estimate holds, we conclude using \eqref{omega_1_omega_2_k_prob_estimate} that for $R\ge C(\eps)$
\begin{equation*}
\P(E_1^\xi)\ge c\exp(-CR^{1+2\eps(d-1)}-CR^{d-2\gamma})
\end{equation*}
as required.
\end{proof}

Theorem~\ref{far_stars_mass_bound_thm} follows from this Corollary by a straightforward application of Proposition~\ref{cylinder_construction_prop} taking $\Vt$ as $V$. We conclude the first construction by proving Proposition~\ref{probability_of_atypical_force_prop}.


\begin{proof}[Proof of Proposition~\ref{probability_of_atypical_force_prop}]
We prove the four statements in the proposition:
\begin{enumerate}
\item Using Proposition~\ref{exp_of_force_empty_box_prop} we see that on the event $\Omega_1$, $F^1(x)$ can be written as $F^1(x)=\E(F(x)\ |\ \Omega_1)$ and satisfies that for each $x\in V_t$,
\begin{equation*}
|F^1(x)_1|\le CR^{-(d-2)}R^{2\eps(d-1)}< CR^{-(d-5/2)}.
\end{equation*}
Noting additionally that $F^1(x)_n=\frac{1}{|x|}\sum_{i=2}^d x_iF^1(x)_i$, it follows from Proposition~\ref{exp_of_force_empty_box_prop} that on the event $\Omega_1$, for each $x\in\partial \Vt$,
\begin{equation*}
F^1(x)_n\ge \frac{1}{M}\sum_{i=2}^d cx_i^2\left(1-\left(\frac{CR^{2\eps}}{R^{2\eps}+R}\right)^{d-1}\right)> cM.
\end{equation*}
\item This follows from Theorem \ref{small_ball_estimate_thm} since $\Vt\subseteq \Vs$ by \eqref{R_and_M_relation}.
\item Let $A:=\R^d\setminus(\Vh\cup U)$. By our assumption \eqref{R_and_M_relation} on the relation between $M$ and $R$, we have $d(\Vt,A)\ge cR$. Hence for each $x\in \Vt$ and each $D>1$, we have by the moderate deviation Theorem~\ref{moderate_deviation_in_ball_thm} that
\begin{equation*}
\begin{split}
\P\left(\max_{y\in B(x,1)} |F(x\ |\ A)|> \frac{D\log^{1/2}(R)}{R^{(d-2)/2}}\right)&\le R^{d^2}\exp(-cD^2\log(R))=\\
&=R^{d^2-cD^2},
\end{split}
\end{equation*}
for $R\ge C$. Since $\Vt$ may be covered by less than $CM^{d-1}R$ balls of radius 1, we obtain by a union bound
\begin{equation*}
\P\left(\max_{y\in B(x,1)} |F(x\ |\ A)|> \frac{D\log^{1/2}(R)}{R^{(d-2)/2}}\right)\le CM^{d-1}R^{d^2+1-cD^2}.
\end{equation*}
Hence we may choose $D$ to be a large enough constant so that this latter probability is less than $\frac{1}{2}$ (using again the relation \eqref{R_and_M_relation}). This proves the claim.
\item Let $\Omega_2':=\{\text{$U$ contains exactly $\vol(U)$ stars}\}$ (recall that $\vol(U)$ is an integer). Note that by Lemma~\ref{poissonlemma},
\begin{equation}\label{prob_of_empty_U_estimate}
\P(\Omega_2')\ge \frac{c}{\sqrt{\vol(U)}}\ge\frac{c}{R^{d/2}}.
\end{equation}
Let $Z_1,\ldots Z_k$ be IID random uniform points in $U$ independent of the Poisson point process (and in particular, independent of $F^4$). Define
\begin{equation*}
F^5(x):=\sum_{i=1}^k \frac{Z_i-x}{|Z_i-x|^d}.
\end{equation*}
Then we have the following equality in distribution: $F^4$ conditioned on $\Omega_2^k$ is equal in distribution to $F^4$ conditioned on $\Omega_2'$ plus $F^5$. This follows directly from the definition \eqref{force_on_bounded_set} of the force $F$. Let us define
\begin{equation*}
\begin{split}
\Omega_6^A&:=\left\{\max_{x\in \Vt} |F^4(x)|\le \frac{A\log^{1/2}(R)}{R^{(d-2)/2}}\right\},\\
\Omega_7^{A,k}&:=\left\{\forall x\in \Vt\quad A\frac{k}{R^{d-1}}>F^5(x)_1>A^{-1}\frac{k}{R^{d-1}}\right\},\\
\Omega_8^{A,k}&:=\left\{\max_{x\in \partial' \Vt} |F^5(x)_n|\le A\right\}.
\end{split}
\end{equation*}
We will show that for large enough $A>1$ (independent of $k$):
\begin{align}
\P(\Omega_6^A)&\ge 1-\frac{c}{R^d},\label{Omega_6_estimate} \\
\P(\Omega_7^{A,k})&=1\label{Omega_7_prob},\\
\P(\Omega_8^{A,k})&\ge 1-CM^{d-2}R\exp(-cR^{d-2})\label{Omega_8_prob}.
\end{align}
Fixing such $A$, we claim that the estimate we want to prove follows from these claims. To see this, first note that by \eqref{prob_of_empty_U_estimate} and \eqref{Omega_6_estimate}, we have $\P(\Omega_6^A,\Omega_2')\ge \frac{\P(\Omega_2')}{2}$ for $R\ge C$. Hence $\P(\Omega_6^A\ |\ \Omega_2')\ge \frac{1}{2}$ for such $R$. Now use the equality in distribution asserted above and the assumption that $C_{11}\log^{1/2}(R)R^{d/2}<k<R^d$ to estimate
\begin{multline*}
\P\Bigg[\left\{\forall x\in \Vt\quad 2A\frac{k}{R^{d-1}}>F^4(x)_1>\frac{A^{-1}k}{2R^{d-1}}\right\}\bigcap\\
\left\{\max_{x\in \partial' \Vt} |F^4(x)_n|\le 2A\right\}\ |\ \Omega_2^k\Bigg]\ge\\
\ge \P(\Omega_6^A\ |\ \Omega_2')\P(\Omega_7^{A,k},\Omega_8^{A,k}) \ge c
\end{multline*}
for $R\ge C(A)$ and $C_{11}\ge 2A^2$. It remains to prove \eqref{Omega_6_estimate}, \eqref{Omega_7_prob} and \eqref{Omega_8_prob}. Estimate \eqref{Omega_6_estimate} follows directly from the moderate deviation Theorem~\ref{moderate_deviation_in_ball_thm} by covering $\Vt$ by less than $R^C$ balls of radius 1. To see \eqref{Omega_7_prob}, fix $x\in \Vt$ and note that for $z\in U$ we have $|z-x|\le CR$ and $z_1-x_1\ge cR$. Hence
\begin{equation*}
F^5(x)_1=\sum_{i=1}^k \frac{Z_{i,1}-x_1}{|Z_i-x|^d}\ge \frac{ck}{R^{d-1}}.
\end{equation*}
Similarly since $|z-x|\ge cR$ and $z_1-x_1\le CR$ for $z\in U$, we obtain $F^5(x)_1\le \frac{Ck}{R^{d-1}}$.

Finally, we prove \eqref{Omega_8_prob}. We start by estimating $\E F^5(x)_n$ for $x\in\partial' \Vt$. Note that by rotational symmetry it is enough to do so for such $x$ with $x_2=M$ and $x_3=\cdots =x_d=0$. Fix such an $x$, and observe that by considering the cancellation in the integrand we have
\begin{equation*}
\begin{split}
\E F^5(x)_n = \E F^5(x)_2 &= \frac{k}{\vol(U)}\int_U \frac{z_2-x_2}{|z-x|^d}dz =\\
&= \frac{k}{\vol(U)}\int_{U\setminus(U_M\cup U'_M)}\frac{z_2-x_2}{|z-x|^d}dz
\end{split}
\end{equation*}
where $U_M:=\{z\in U\ | z_2\ge M\}$ and $U'_M$ is the reflection of $U_M$ in the $\{z_2=M\}$ hyperplane (see Figure~\ref{fig-cylindersimage}).
\begin{center}
\begin{figure}[ht] 
\center{\resizebox{200pt}{!}{\includegraphics{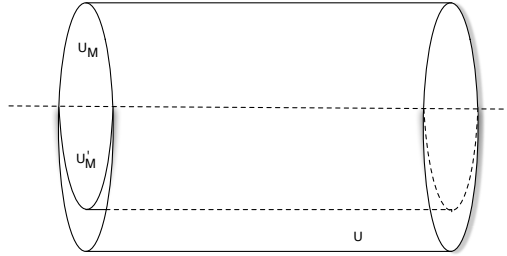}}}
{\it
\caption{$U_M$ and $U'_M$.\label{fig-cylindersimage}}
}
\end{figure}
\end{center}
Observe that $U_M\cup U'_M$ is contained in $U$ and contains a translated copy of $\Cyl{R}{\eta R-M}$. From this and the assumption that $k<M^{-1}R^d$, we obtain
\begin{equation}\label{F_5_expectation_estimate}
\begin{split}
\left|\E F^5(x)_n\right| &\le \frac{Ck}{\vol(U)R^{d-1}}\vol(U\setminus(U_M\cup U'_M))\le\\
&\le \frac{CkR^{d-1}M}{\vol(U)R^{d-1}}\le \frac{CkM}{R^d}\le C.
\end{split}
\end{equation}
We use Bernstein-Hoeffding's inequality \cite{hoeffding} to bound the deviation from the expectation. Using again that for $z\in U$ we have $\left|\frac{z_2-x_2}{|z-x|^d}\right|\le \frac{C}{R^{d-1}}$, we deduce that
\begin{equation*}
\P(|F^5(x)-\E F^5(x)|\ge t)\le C\exp\left(-\frac{ct^2R^{2d-2}}{k}\right)\le C\exp(-ct^2R^{d-2}).
\end{equation*}
Combining this with \eqref{F_5_expectation_estimate} we see that
\begin{equation*}
\P\left(|F^5(x)|\ge \rho\right)\le C\exp(-c\rho^2R^{d-2})
\end{equation*}
for large enough $\rho$. Fix such a $\rho$. Now use estimate \eqref{g_and_Dg_estimates} to deduce that for every $x\in \Vt$, $|D_1 F^5(x)|\le \frac{Ck}{R^d}\le C$. Since we may cover $\partial' \Vt$ by not more than $CM^{d-2}R$ balls of radius $1$, a union bound gives
\begin{equation*}
\P(\max_{x\in \partial' \Vt} |F^5(x)_n|\ge \rho+C)\le CM^{d-2}R\exp(-c\rho^2R^{d-2})
\end{equation*}
finishing the proof of \eqref{Omega_8_prob} and thus the proposition.


\end{enumerate}
\end{proof}
\end{subsection}

\begin{subsection}{Second lower bound construction - Wormhole}\label{second_lower_bound_section}
In this section we shall prove:
\begin{theorem}\label{near_stars_mass_bound_thm}
For all dimensions $d\ge 4$ and $0<\eps<\frac{1}{10d}$, there exist $C,c,C(\eps,\gamma)>0$ such that for all $0\le\gamma\le 2$ and $R\ge C(\eps,\gamma)$,
\begin{equation*}
\P(Z_R>\exp(-CR^\gamma)/R^C)\ge c\exp(-CR^{1+2\eps(d-1)}-CR^{1+\frac{2-\gamma}{d-2}+2\eps(d-3)}\log R)
\end{equation*}
\end{theorem}
In addition to the common parts of the constructions, we define (see also Figure~\ref{fig-cylinder-second})
\begin{equation*}
\begin{split}
&W:=\lambda R^{-\frac{2-\gamma}{d-2}+2\eps},\\
&U := \Cyl{R}{W},
\end{split}
\end{equation*}
where $0<\lambda<1$ is a small constant depending only on $d$ whose value will be determined in the sequel. We always assume that $U\subseteq \Vs$, which occurs for small enough $\lambda$. We also consider a layer around the boundary of the set $U$,
\begin{equation*}
\begin{split}
\rho&:=R^{-3d},\\
\partial'_{\rho} U&:=\{x\ |\ d(x,\partial' U)\le \rho\}.
\end{split}
\end{equation*}

We divide the force into four parts
\begin{equation*}
F(x)=\underbrace{F(x\ |\ 2\Vs\setminus \partial'_{\rho} U)}_{F^1(x)} + \underbrace{F(x\ |\ \Vh\setminus 2\Vs)}_{F^2(x)} + \underbrace{F(x\ |\ \R^d\setminus \Vh)}_{F^3(x)} + \underbrace{F(x\ |\ \partial'_{\rho} U)}_{F^4(x)}
\end{equation*}
and further divide
\begin{equation*}
F^4(x) = \underbrace{\sum_{z\in\CZ\cap\partial'_{\rho} U}\frac{z-x}{|z-x|^d}}_{F^{4,1}(x)}-\underbrace{\int_{\partial'_{\rho} U} \frac{z-x}{|z-x|^d}dz}_{F^{4,2}(x)}.
\end{equation*}

Define the event
\begin{equation*}
\Omega_1:=\{2\Vs\setminus \partial'_{\rho} U\text{ contains no stars}\}
\end{equation*}
and note that
\begin{equation}\label{Omega_1_prob_estimate}
\P(\Omega_1)\ge \exp(-\vol(2\Vs))\ge \exp(-CR^{1+2\eps(d-1)}).
\end{equation}
As in the previous section, we lower bound the probabilities that $F^2$ and $F^3$ give a negligible contribution to the force uniformly on $\frac{1}{3} U$ and we estimate the contribution of $F^1-F^{4,2}$.
\begin{proposition}\label{probability_of_atypical_force_prop2}
For $0<\eps<\frac{1}{10d}$, there exist $C_{10},C_{11},C,c,C(\eps),c(\eps)>0$ such that if $R\ge C(\eps)$ then:
\begin{enumerate}
\item On the event $\Omega_1$, we have deterministically that
\begin{equation*}
\begin{split}
\forall &x\in \frac{1}{3} U\quad |F^1(x)_1-F^{4,2}(x)_1|<\frac{C}{R^{d-5/2}},\\
\forall &x\in \frac{1}{3}\partial' U\quad F^1(x)_n-F^{4,2}(x)_n>cW.
\end{split}
\end{equation*}
\item \begin{equation*}
\P\left(\max_{x\in \frac{1}{3} U} |F^2(x)|\le \frac{1}{R^d}\right)\ge c(\eps)\exp(-C(\eps)R^{1+\eps(d-2)}\log R).
\end{equation*}
\item \begin{equation*}
\P\left(\max_{x\in \frac{1}{3} U} |F^3(x)|\le \frac{C_{10}\log^{1/2}(R)}{R^{(d-2)/2}}\right)\ge \frac{1}{2}.
\end{equation*}
\end{enumerate}
\end{proposition}
The proof of the proposition is the same as the proofs of parts 1 to 3 of Proposition \ref{probability_of_atypical_force_prop} with $\frac{1}{3} U$ replacing $\Vt$ and with $\R^d\setminus \Vh$ replacing $\R^d\setminus (\Vh\cup U)$ in part 3.

It remains to control $F^{4,1}$. For a finite set $A\subseteq\text{Closure}(\partial' U)$ let
\begin{eqnarray*}
\Omega_2^A&:=&\Big\{\text{There exists a bijection $T:A\to(\CZ\cap\partial'_\rho U)$}\\ & & \ \ \,\text{with $d(x,T(x))\le\rho$\ \ $\forall x\in A$}
\Big\}.
\end{eqnarray*}
We note that
\begin{lemma}\label{Omega_2_A_prob_estimate}
Let $X$ be a random variable with $\Poisson(\vol(\partial'_\rho U))$ distribution.
For any finite set $A\subseteq \text{Closure}(\partial' U)$,
\begin{equation*}
\P(\Omega_2^A)\ge \P(X=|A|)\left(\frac{\vol(B(0,\rho))}{\vol(\partial'_\rho U)}\right)^{|A|}
\end{equation*}
\end{lemma}
\begin{proof}
To prove the lemma enumerate the points in $A$ by $x_1, \ldots, x_{|A|}$ and note that $\Omega_2^A$ occurs if there are exactly $|A|$ stars in $\partial'_\rho U$ and if the $i$-th star falls in $B(x_i,\rho)$.
\end{proof}
The main proposition of this section is
\begin{proposition}\label{discrete_approx_of_special_cubature_prop}
For $d\ge 4$, there exist $C,c,C(\eps)>0$ and a finite set $A\subseteq\text{Closure}(\partial' U)$ with $|A|\le CR^{1+\frac{2-\gamma}{d-2}+2\eps(d-3)}$ such that if $R\ge C(\eps)$ and if $\Omega_2^A$ occurred then
\begin{enumerate}
\item For all $x\in\frac{1}{3} U$ we have $c\lambda R^{1-\gamma}\le F^{4,1}(x)_1\le C\lambda R^{1-\gamma}$.
\item For all $x\in\frac{1}{3} \partial 'U$ we have $|F^{4,1}(x)_n|\le C\lambda^{d-2} R^{-(d-2)+\frac{2-\gamma}{d-2}+2\eps(d-3)}$.
\end{enumerate}
\end{proposition}
\begin{corollary}\label{force_in_cylinder_for_large_d_cor}
We may choose $\xi, C,c>0$, the constant $0<\lambda<1$ from the definition of $W$ (independently of $R$ and $\eps$) and $C(\eps,\gamma)>0$ such that for all $d\ge 4$ and $0\le\gamma<2$, letting
\begin{equation*}
E_1^{\xi}:=\bigg\{\forall x\in \frac{1}{3} U\quad \xi R^{1-\gamma}\ge F(x)_1\ge \xi^{-1} R^{1-\gamma}\bigg\}\cap\bigg\{\min_{x\in\frac{1}{3}\partial' U} F(x)_n> 0\bigg\},
\end{equation*}
we have for $R\ge C(\eps,\gamma)$ that
\begin{equation}\label{large_forces_in_U_prob_estimate}
\P(E_1^\xi)\ge c\exp(-CR^{1+2\eps(d-1)}-CR^{1+\frac{2-\gamma}{d-2}+2\eps(d-3)}\log R).
\end{equation}
\end{corollary}
We remark that the requirement that $R\ge C(\eps,\gamma)$ may be weakened to $R\ge C(\eps)$ and the requirement $0\le\gamma<2$ strengthened to $0\le\gamma\le 2$ by choosing the parameters a little differently in Proposition \ref{discrete_approx_of_special_cubature_prop} and obtaining $F^4(x)_1$ of a larger order of magnitude for a higher probabilistic cost. Since this complicates the proof slightly and does not contribute much to the final result we do not describe this improvement.
\begin{proof}[Proof of Corollary \ref{force_in_cylinder_for_large_d_cor}]
We let $A$ be the set from Proposition \ref{discrete_approx_of_special_cubature_prop} and note that by \eqref{Omega_1_prob_estimate} and Lemma \ref{Omega_2_A_prob_estimate} we have for $R\ge C$ that
\begin{equation*}
\P(\Omega_1, \Omega_2^A)\ge c\exp(-CR^{1+2\eps(d-1)}-CR^{1+\frac{2-\gamma}{d-2}+2\eps(d-3)}\log R)
\end{equation*}
Note that if we let $\mathcal F_1:=\sigma(F^1, \Omega_1), \mathcal F_2:=\sigma(F^2), \mathcal F_3:=\sigma(F^3)$ and $\mathcal F_4:=\sigma(F^4, \Omega_2^A)$ (where as before $\sigma(\cdot)$ denotes the generated $\sigma$-field), then 
the $\mathcal F_i$'s are independent $\sigma$-fields.
Hence, by Proposition \ref{probability_of_atypical_force_prop2} (combining all 3 parts) and Proposition \ref{discrete_approx_of_special_cubature_prop} we see that for large enough $\xi$ and $R\ge C(\eps)$ we have that
\begin{equation*}
\P(E_1^\xi\ |\ \Omega_1, \Omega_2^A)\ge c(\eps)\exp(-C(\eps)R^{1+\eps(d-2)}\log R),
\end{equation*}
provided that the error terms affecting $F(x)_1$ for $x\in\frac{1}{3} U$ and $F(x)_n$ for $x\in\frac{1}{3}\partial' U$ do not dominate the main terms. This occurs, for example, when
\begin{equation*}
\begin{split}
cW&>2\max(R^{-d}, \frac{C_{10}\log^{1/2}(R)}{R^{(d-2)/2}}, C\lambda^{d-2} R^{-(d-2)+\frac{2-\gamma}{d-2}+2\eps(d-3)}),\\
c\lambda R^{1-\gamma}&>2\max(\frac{C}{R^{d-5/2}}, R^{-d}, \frac{C_{10}\log^{1/2}(R)}{R^{(d-2)/2}})
\end{split}
\end{equation*}
which when $d\ge 5$ happens for $R\ge C(\eps)$ and for $d=4$ happens when $\lambda$ is sufficiently small and $R\ge C(\eps,\gamma)$. This concludes the proof.
\end{proof}

Theorem~\ref{near_stars_mass_bound_thm} for $d\ge 4$ and $0\le
\gamma<2$ follows from this Corollary by a straightforward application of Proposition~\ref{cylinder_construction_prop} taking $\frac{1}{3} U$ as $V$. The case of $\gamma=2$ follows from the other cases
since $\{Z_R>\exp(-CR^2)/R^C\}\supseteq\{Z_R>\exp(-CR^\gamma)/R^C\}$
for all $0\le \gamma<2$. Proposition~\ref{discrete_approx_of_special_cubature_prop} will be proved over the next 3 subsections.


\begin{subsubsection}{Continuous version of Proposition \ref{discrete_approx_of_special_cubature_prop}}\label{continuous_version_of_construction_sec}
In this section we shall formulate and prove a continuous version of Proposition \ref{discrete_approx_of_special_cubature_prop}. The proposition will then be proved in the next section by approximating this continuous version.

Recall the definition of $\nu_{L,W}$ from before Theorem~\ref{special_cubature_on_cyl_thm}. Set
\begin{equation*}
\begin{split}
\beta&:=R^{(2-\gamma)\frac{d-1}{d-2}-2\eps},\\
\nu&:=\beta\nu_{R,W}.
\end{split}
\end{equation*}
Note that $\nu$ is supported on Closure($\partial' U$), but gives full mass to $\partial' U$. We define the ``gravitational force from the mass distribution $\nu$'' as
\begin{equation*}
G(x):=\int_{\partial' U} \frac{z-x}{|z-x|^d}d\nu(z).
\end{equation*}
We note that if the stars in $\partial'_\rho U$ were ``placed according to the distribution $\nu$'', then $F^{4,1}(x)$ would equal $G(x)$.

\begin{lemma}\label{equilibrium_measure_on_cyl}
\mbox{}
\begin{enumerate}
\item Fix $M>0$, let $V:=\{x\ |\ x_2^2+x_3^2+\cdots x_d^2\le M^2\}$ and define
\begin{equation*}
H'_M(x):=\int_{\partial V} \frac{z-x}{|z-x|^d}d\sigma_{d-1}(z)
\end{equation*}
then for each $x\in V^{\circ}$, the interior of $V$, the integral defining $H'_M(x)$ converges absolutely and $H'_M(x)=0$.
\item There exists $C>0$ such that 
for each $x\in\frac{1}{3} \partial' U$,
\begin{equation*}
|G(x)_n|\le C\beta W^{d-2}R^{-(d-2)}.
\end{equation*}
\end{enumerate}
\end{lemma}

Note that Claim 1. above says that for an infinite cylinder $V$, the surface area measure $({\sigma_{d-1}})_{\big|V}$
is the potential-theoretic equilibrium measure.

We continue with an estimate of the first component of the gravitational force $G$ from $\nu$ in $U$.
\begin{lemma}\label{continuous_rightward_force_lemma}
For dimensions $d\ge 4$ there exist $C,c>0$ such that if $R\ge C$, for each $x\in\frac{1}{3} U$
\begin{equation*}
c\lambda R^{1-\gamma} \le G(x)_1\le C\lambda R^{1-\gamma}.
\end{equation*}
\end{lemma}
\begin{proof}[Proof of Lemma \ref{equilibrium_measure_on_cyl}]
\begin{enumerate}
\item Fix $x\in V^{\circ}$. $H'_M(x)$ converges absolutely since the contribution to the norm of the integral from all the $z$ with $|z_1-x_1|=L$ is less than $CM^{d-2}L^{-(d-1)}$. We have $H'_M(x)_1=0$ by symmetry. Finally, $H'_M(x)_n=0$ follows from rotational symmetry and the divergence theorem.
\item Fix $x\in\frac{1}{3}\partial' U$. By rotating the coordinate system we may assume WLOG that $x$ is such that $G(x)_n=G(x)_2$. Let $\tilde{U}:=\{z\ |\ \sqrt{z_2^2+z_3^2+\cdots z_d^2}=W\}$, $P_1:=\{z\in \tilde{U}\ |\ |x_1-z_1|\le \frac{R}{2}\}$ and $P_2:=\{z\in \tilde{U}\ |\ |x_1-z_1|> \frac{R}{2}\}$. By the linearity of the density of $\nu$ we have
\begin{equation*}
\begin{split}
G(x)_n &= \int_{P_1} \frac{z_2-x_2}{|z-x|^d}d\nu(z) + \int_{\partial' U\setminus P_1} \frac{z_2-x_2}{|z-x|^d}d\nu(z) =\\
&= C_x\beta\int_{P_1} \frac{z_2-x_2}{|z-x|^d}d\sigma_{d-1}(z) + \int_{\partial' U\setminus P_1} \frac{z_2-x_2}{|z-x|^d}d\nu(z)
\end{split}
\end{equation*}
for some $1\le C_x\le 2$ and by the previous part,
\begin{equation*}
\begin{split}
&C_x\beta\left|\int_{P_1} \frac{z_2-x_2}{|z-x|^d}d\sigma_{d-1}(z)\right|=C_x\beta\left|\int_{P_2} \frac{z_2-x_2}{|z-x|^d}d\sigma_{d-1}(z)\right|\le \\
&\le C_x\beta\int_{R/2}^\infty CW^{d-2}L^{-(d-1)}dL\le C\beta W^{d-2}R^{-(d-2)}.
\end{split}
\end{equation*}
Similarly
\begin{equation*}
\left|\int_{\partial' U\setminus P} \frac{z-x}{|z-x|^d}d\nu(z)\right|\le C\beta\int_{R/2}^\infty CW^{d-2}L^{-(d-1)}dL \le C\beta W^{d-2}R^{-(d-2)},
\end{equation*}
as required.
\end{enumerate}
\end{proof}

\begin{proof}[Proof of Lemma \ref{continuous_rightward_force_lemma}] Fix $x\in\frac{1}{3} U$ and define $C_t:=\partial' U\cap \{x_1=t\}$. By definition
\begin{equation*}
\begin{split}
G(x)_1 &= \beta\int_{-R}^R \left(1+\frac{t+R}{2R}\right)\int_{C_t}\frac{z_1-x_1}{|z-x|^d}d\sigma_{d-2}(z)dt=\\
&=\beta\left(\underbrace{\int_{-R}^{x_1-R/2} \cdots + \int_{x_1+R/2}^R\cdots}_{G^1(x)} + \underbrace{\int_{x_1-R/2}^{x_1+R/2}\cdots}_{G^2(x)}\right).
\end{split}
\end{equation*}
We first estimate $G^1(x)$ by
\begin{equation*}
|G^1(x)|\le CW^{d-2}R^{-(d-2)}.
\end{equation*}
We continue by noting that the contribution to $G^2(x)$ from $C_{x_1-s}$ cancels with some of the contribution from $C_{x_1+s}$, giving
\begin{equation*}
G^2(x)=\int_{0}^{R/2} \frac{s}{R}\int_{C_{x_1+s}} \frac{z_1-x_1}{|z-x|^d}d\sigma_{d-2}(z)ds=\underbrace{\int_0^{W}\cdots}_{G^3(x)}+\underbrace{\int_{W}^{R/2}\cdots}_{G^4(x)}.
\end{equation*}
We note that if $z\in C_{x_1+s}$ for $0\le s\le W$ then $cW^{-d}s\le \frac{z_1-x_1}{|z-x|^d}\le CW^{-d}s$. Since
\begin{equation*}
\int_0^{W}\frac{s}{R} \int_{C_{x_1+s}} W^{-d}s d\sigma_{d-2}(z)ds = c WR^{-1}
\end{equation*}
for some constant $c>0$, we deduce that
\begin{equation*}
cWR^{-1}\le G^3(x)\le CWR^{-1}.
\end{equation*}
Similarly, if $z\in C_{x_1+s}$ for $W\le s\le R/2$ we have $cs^{-(d-1)}\le \frac{z_1-x_1}{|z-x|^d}\le Cs^{-(d-1)}$. Since when $d\ge 4$ and $R\ge C$ we have
\begin{equation*}
\int_{W}^{R/2}\frac{s}{R} \int_{C_{x_1+s}} s^{-(d-1)} d\sigma_{d-2}(z)ds = c WR^{-1}
\end{equation*}
for some constant $c>0$, we deduce that
\begin{equation*}
cWR^{-1}\le G^4(x)\le CWR^{-1}.
\end{equation*}
Putting all the above estimates together and noting that when $R\ge C$ we have $|G^1(x)|\le \frac{1}{2}|G^2(x)|$ for all $d\ge 4$, we obtain
\begin{equation*}
c\beta WR^{-1}\le G(x)_1\le C\beta WR^{-1}
\end{equation*}
which concludes the proof since $\beta WR^{-1}=\lambda R^{1-\gamma}$.
\end{proof}
\end{subsubsection}

\begin{subsubsection}{Discrete Approximation}
In this section we approximate the continuous distribution $\nu$ of the previous section by a measure $\nu'$ of the form $\nu'=\sum_{z\in A} \delta_z$ for a set $A\subseteq \text{Closure}(\partial' U)$. Our approximation will be such that the force exerted by $\nu$ and by $\nu'$ on points in $\frac{1}{3} U$ will remain approximately the same. This is done by using Theorem~\ref{special_cubature_on_cyl_thm} and Proposition~\ref{Taylor_expansion_use_prop}.

We introduce parameters ($W$ was already introduced)
\begin{align*}
L&=\frac{R}{2},& W&=\lambda R^{-\frac{2-\gamma}{d-2}+2\eps},\\ 
r&=\frac{W}{100},& t&=R^{-d},\\
\tau&=\eta WR^{-\eps},& \delta&=r^{d-1}\left(\frac{r}{2d+r}\right)^k tR^{-\eps d},\\ 
k&=M(\eps),& n &= \beta\tau^{d-1},
\end{align*}
where $M(\eps)>0$ is a constant depending only on $\eps$ and $d$, chosen large enough for the following calculations, and $\frac{1}{2}<\eta<1$ is chosen so that $n=n_1^{d-1}$ for an integer $n_1$. 

We recall that in the notation of Theorem~\ref{special_cubature_on_cyl_thm}, Closure($\partial' U$) is the cylinder $P_{2L,W}$. We use the theorem for the measure $\nu_{2L,W}$ with the above parameters $L,W,\tau,\delta,k$ and $n$ (one checks that if $R\ge C(\eps)$, this choice of $n$ satisfies part (III) of the theorem) to obtain $D_1,\ldots, D_{K}\subseteq \text{Closure}(\partial' U)$ and points $(w_{D_i,j})_{j=1}^n\subseteq D_i$ satisfying the properties of the theorem.

We now fix $1\le i\le K$ and define the measure $\nu'_i:=\sum_{j=1}^n \delta_{w_{D_i,j}}$ whose support is in $D_i$. By part (III) of the theorem we have for each $h:\R^d\to\R$ which is of the form $h(w)=(w-y)^\alpha$ for some $y\in\text{Closure}(\partial' U)$ and some multi-index $\alpha$ with $|\alpha|\le k$ that
\begin{equation*}
\left|\int h(w)d\nu'_i(w) - \frac{n}{\nu_{2L,W}(D_i)}\int_{D_i} h(w)d\nu_{2L,W}(w)\right|\le \delta n.
\end{equation*}
But, by part (II) of the theorem, $\frac{n}{\nu_{2L,W}(D_i)}\nu_{2L,W} = \nu$. Hence
\begin{equation}\label{moment_approximation_on_Di}
\left|\int h(w)d\nu'_i(w) - \int_{D_i} h(w)d\nu(w)\right|\le \delta n.
\end{equation}
We now apply Proposition \ref{Taylor_expansion_use_prop} with the set $U$ of the proposition being $D_i$ and with the variables $Y_1,\ldots, Y_n$ of the proposition being IID samples from $\nu$ restricted to $D_i$ and normalized to be a probability measure. Fix a point $y\in D_i$ and let $M_j:=P_k^d(Y_j-y)$. We note that if $Y_j=w_{D_i,j}$ for $1\le j\le n$ then by \eqref{moment_approximation_on_Di} we have
\begin{equation*}
\left|\sum_{j=1}^n M_j - \E\sum_{j=1}^n M_j\right|=\left|\int h(w)d\nu'_i(w) - \int_{D_i} h(w)d\nu(w)\right|\le \delta n.
\end{equation*}
Since we also have for $R\ge C(\eps)$, $d\ge 4$ and the above choices of $r$ and $t$ that $\delta n\le \frac{c_{30}tr^{d-1}}{\polydim(k,d)^{1/2}}\left(\frac{r}{2d+r}\right)^k$ and conditions \eqref{r_condition} and \eqref{t_condition} hold (by part (II) of Theorem~\ref{special_cubature_on_cyl_thm}, $\sup_{z\in D_i} |z-y|\le C\tau$), we deduce from the proposition that
\begin{equation*}
\max_{x\in \frac{1}{3} U} \left|\sum_{j=1}^n g(Y_j-x) - \E\sum_{j=1}^n g(Y_j-x)\right|\le t
\end{equation*}
when $Y_j=w_{D_i,j}$ for $1\le j\le n$. In other words
\begin{equation*}
\max_{x\in \frac{1}{3} U} \left|\int \frac{w-x}{|w-x|^d}d\nu'_i(w) - \int_{D_i}\frac{w-x}{|w-x|^d}d\nu(w)\right|\le t.
\end{equation*}
Finally, defining the set $A:=\{w_{D_i,j}\}_{\substack{i=1\ldots K\\j=1\ldots n}}$ and the measure $\nu'=\sum_{i=1}^{K} \nu'_i=\sum_{w\in A} \delta_w$ we obtain
\begin{equation}\label{force_approx_on_D_i_union}
\max_{x\in \frac{1}{3} U} \left|\int \frac{w-x}{|w-x|^d}d\nu'(w) - \int_{\cup_{i=1}^{K}D_i}\frac{w-x}{|w-x|^d}d\nu(w)\right|\le tK.
\end{equation}
By part (II) of Theorem~\ref{special_cubature_on_cyl_thm}, we have (for $d\ge 4$)
\begin{equation}\label{K_3_estimate}
K\le CLW^{d-2}\tau^{-(d-1)}\le CR^{1+\frac{2-\gamma}{d-2}-\eps(d-1)}\le CR^2.
\end{equation}
And also
\begin{equation}\label{number_of_stars_estimate}
|A| = nK\le C\beta W^{d-2}R \le CR^{1+\frac{2-\gamma}{d-2}+2\eps(d-3)}.
\end{equation}
To end this section, we prove
\begin{lemma}\label{discrete_approximation_lemma}
There exists $C,C(\eps)>0$ such that if $R\ge C(\eps)$ then
\begin{equation*}
\max_{x\in \frac{1}{3} U} \left|\int \frac{w-x}{|w-x|^d}d\nu'(w) - \int_{\partial' U} \frac{w-x}{|w-x|^d}d\nu(w)\right|\le C\beta W^{d-2}R^{-(d-2)}.
\end{equation*}
\end{lemma}
\begin{proof}
By \eqref{force_approx_on_D_i_union} and \eqref{K_3_estimate} we have
\begin{equation*}
\max_{x\in \frac{1}{3} U} \left|\int \frac{w-x}{|w-x|^d}d\nu'(w) - \int_{\cup_{i=1}^{K}D_i}\frac{w-x}{|w-x|^d}d\nu(w)\right|\le CR^{-(d-2)}.
\end{equation*}
Since for $R\ge C(\eps)$ we have $\beta W^{d-2}\ge 1$, it is enough to prove that
\begin{equation*}
\max_{x\in \frac{1}{3} U}\left|\int_{\partial' U\setminus \cup_{i=1}^{K} D_i} \frac{w-x}{|w-x|^d}d\nu(w)\right|\le C\beta W^{d-2}R^{-(d-2)}.
\end{equation*}
By part (I) of Theorem~\ref{special_cubature_on_cyl_thm} we know that up to $\nu$-measure 0, $\partial' U\setminus \cup_{i=1}^{K} D_i$ is contained in $\{x\in\partial' U\ |\ |x_1|\ge \frac{R}{2}\}$. Hence, just as in the previous section,
\begin{equation*}
\begin{split}
\max_{x\in \frac{1}{3} U}\left|\int_{\partial' U\setminus \cup_{i=1}^{K} D_i} \frac{w-x}{|w-x|^d}d\nu(w)\right|&\le C\beta\int_{R/2}^\infty W^{d-2}L^{-(d-1)}dL\le\\
&\le C\beta W^{d-2}R^{-(d-2)}.\qedhere
\end{split}
\end{equation*}
\end{proof}
\end{subsubsection}

\begin{subsubsection}{Proof of Proposition \ref{discrete_approx_of_special_cubature_prop}}
For the set $A$, we take the set constructed in the previous section. It remains to show that it fulfills the properties in the proposition. Assume that $\Omega_2^A$ occurred and enumerate the points in $A$ by $w_1,\ldots, w_{|A|}$ and the stars in $\partial'_{\rho} U$ by $Y_1,\ldots, Y_{|A|}$ in such a way that $d(w_i,Y_i)\le \rho$ for all $i$. By definition we have that
\begin{equation*}
F^{4,1}(x) = \sum_{i=1}^{|A|} \frac{Y_i-x}{|Y_i-x|^d}.
\end{equation*}
Fix $x\in\frac{1}{3} U$. We recall from \eqref{g_and_Dg_estimates} that $|D_1 g(x)|=\left|D_1 \frac{x}{|x|^d}\right|\le C|x|^{-d}$. We now estimate
\begin{equation*}
\left|\sum_{i=1}^{|A|} \frac{Y_i-x}{|Y_i-x|^d} - \sum_{i=1}^{|A|} \frac{w_i-x}{|w_i-x|^d}\right|\le CW^{-d}\rho|A|\le CR^{-d}
\end{equation*}
for $R\ge C$, by our choice of $\rho$ and by \eqref{number_of_stars_estimate}.

It follows from this and Lemma \ref{discrete_approximation_lemma} that
\begin{equation*}
\max_{x\in\frac{1}{3} U}\left|F^{4,1}(x)-G(x)\right|\le C\beta W^{d-2}R^{-(d-2)}.
\end{equation*}
Since by Lemmas \ref{equilibrium_measure_on_cyl} and \ref{continuous_rightward_force_lemma} we have for each $x\in\frac{1}{3} U$
\begin{align*}
c\lambda R^{1-\gamma} \le &G(x)_1\le C\lambda R^{1-\gamma}
\end{align*}
and for $x\in\frac{1}{3}\partial' U$
\begin{equation*}
|G(x)_n|\le C\beta W^{d-2}R^{-(d-2)}
\end{equation*}
the proposition is proven.
\qed
\end{subsubsection}
\end{subsection}
\end{section}
\begin{section}{Acknowledgments}
We thank Nir Lev for referring us to the book of Stein and explaining the relevance of oscillatory integrals to the proof of Theorem~\ref{positive_density_thm}. We also thank Boris Tsirelson and Mikhail Sodin for several useful conversations, in particular concerning approximation of continuous measures with discrete ones and finally we thank Greg Kuperberg and Sasha Sodin for useful discussions on cubatures.
\end{section}

Sourav Chatterjee \\
Department of Statistics \\
367 Evans Hall \\
The University of California \\
Berkeley, CA 94720-3860, USA \\
\texttt{sourav@stat.berkeley.edu}

\bigskip \noindent
Ron Peled\\
Courant Institute of Mathematical Sciences \\
251 Mercer St. \\
New York University \\
New York, NY 10012-1185, USA \\
\texttt{peled@cims.nyu.edu}

\bigskip \noindent
Yuval Peres \\
Microsoft Research \\
One Microsoft way \\
Redmond, WA 98052-6399, USA \\
\texttt{peres@microsoft.com}

\bigskip\noindent
Dan Romik \\
Einstein Institute of Mathematics \\
Hebrew University of Jerusalem \\
Givat Ram, Jerusalem 91904, Israel \\
\texttt{romik@math.huji.ac.il}

\begin{thebibliography}{99}
\bibitem{A89} Arnol'd V. I.. Mathematical Methods of Classical Mechanics.
Springer-Verlag, New York, 1989.
\bibitem{CPPR07}
Chatterjee Sourav, Peled Ron, Peres Yuval and Romik Dan. Gravitational allocation to Poisson points. To appear in \emph{Annals of Mathematics}. Preprint at
``http://arxiv.org/abs/math/0611886''.
\bibitem{hoeffding} Hoeffding W. Probability inequalities for sums of bounded random variables. \textit{J. Amer. Stat. Soc.} 58 (1963), 13--30.
\bibitem{NSV07}
Nazarov Fedor, Sodin Mikhail, Volberg Alexander.
Transportation to random zeroes by the gradient flow. \emph{Geometric and Functional Analysis} Vol 17-3, 887-935, 2007 (An older version 1 can be found in
``http://www.arxiv.org/abs/math/0510654v1''.)
\bibitem{P09}
Peled Ron. Simple Universal Bounds for Chebyshev-Type Quadratures. Preprint at
``http://arxiv.org/abs/0903.4625''.
\bibitem{ST06}
Sodin Mikhail and Tsirelson Boris. Random complex zeroes II: Perturbed lattice. \textit{Israel J. Math.} 152 (2006), 105--124.
\end{thebibliography}
\end{document}